\documentclass[12pt]{amsart}
\usepackage{amsmath,amsthm,amssymb}
\theoremstyle{plain}
\newtheorem{thm}{Theorem}[section]
\newtheorem{defi}[thm]{Definition}
\newtheorem{prop}[thm]{Proposition}
\newtheorem{lem}[thm]{Lemma}
\newtheorem{cor}[thm]{Corollary}
\theoremstyle{remark}
\newtheorem{rk}[thm]{\bf Remark}

\numberwithin{equation}{section} 
\newcommand{\bean}{\begin{eqnarray*}}
\newcommand{\eean}{\end{eqnarray*}}
\newcommand{\be}{\begin{equation}}

\newcommand{\ee}{\end{equation}}
\newcommand{\bd}{\begin{displaymath}}
\newcommand{\ed}{\end{displaymath}}

\newcommand{\beq}{\begin{equation}}
\newcommand{\eeq}{\end{equation}}
\newcommand{\bea}{\begin{eqnarray}}
\newcommand{\eea}{\end{eqnarray}}

\newcommand{\abs}[1]{\left\vert{#1}\right\vert}

\newcommand{\R}{\mathbb{R}}
\newcommand{\e}{\varepsilon}

\begin{document}

\author[K. Chen and P. Sternberg]{Ko-Shin Chen and Peter Sternberg
}

\title[Ginzburg-Landau and Gross-Pitaevskii vortices ]{Dynamics of Ginzburg-Landau and Gross-Pitaevskii vortices on Manifolds } \maketitle

\begin{center}
Department of Mathematics,\\ Indiana University,
Bloomington, IN 47405, USA\\
koshchen@indiana.edu, sternber@indiana.edu
\end{center}
\begin{center}
December 5, 2012
\end{center}
\vskip.5in
\begin{abstract}
We consider the dissipative heat flow and conservative Gross-Pitaevskii dynamics associated
with the Ginzburg-Landau energy
\begin{equation*}
    E_\varepsilon(u) = \int_{\mathcal M} \frac{|\nabla_g u|^2}{2} + \frac{(1-|u|^2)^2}{4\varepsilon^2} dv_g
\end{equation*}
 posed on a
Riemannian $2$-manifold $\mathcal{M}$ endowed with a metric $g$. In the $\e \to 0$ limit, we show the vortices
of the solutions to these two problems evolve according to the gradient flow and Hamiltonian point-vortex flow respectively, associated with the renormalized
energy on $\mathcal{M}.$ For the heat flow, we then specialize to the case where $\mathcal{M}=S^2$ and study the limiting system of ODE's and establish
an annihilation result. Finally, for the Ginzburg-Landau heat flow on $S^2$, we derive some weighted energy identities.
\end{abstract}
\section{Introduction}
There is a rich and well-developed theory to describe the motion of vortices arising in
both the heat flow and Schr\"{o}dinger dynamics associated with the Ginzburg-Landau energy.
The setting for much of this work has been on a bounded planar domain or all of $\R^2$, though by now
there have also been several efforts to understand such flows in $\R^n$ for $n>2$.
In this article, however, we pose these problems on a compact, simply-connected 2-manifold. The motivation is
to see how the geometry, and in particular the curvature of the underlying surface may impact on the motion laws
and stability of collections of vortices in these dissipative and dispersive settings.

To this end, we let $\mathcal M$ be a smooth, simply-connected compact 2-manifold without boundary equipped with a metric $g$ and then define the Ginzburg-Landau energy $E_\varepsilon$ on $\mathcal M$ for $u: {\mathcal M} \rightarrow \mathbb C$ by
\begin{equation} \label{GLEintro}
    E_\varepsilon(u) = \int_{\mathcal M} \frac{|\nabla_g u|^2}{2} + \frac{(1-|u|^2)^2}{4\varepsilon^2} dv_g.
\end{equation}
This is a simplified version of the full energy which includes magnetic effects, originally introduced by Ginzburg and Landau \cite{GL} as a model to describe
superconductivity. To $E_\e$ one can then naturally associate the Ginzburg-Landau heat flow
\begin{equation} \label{DGLintro}
    \left\{\begin{array}{ll}
                 u_t^{\varepsilon} - \bigtriangleup _{\mathcal M} u^{\varepsilon} = \frac{u^{\varepsilon}}{\varepsilon^2}(1-|u^{\varepsilon}|^2) & \mbox{in $\mathcal M \times \mathbb R_+$} \\
                 u^{\varepsilon} = u_0^{\varepsilon} & \mbox{on $\mathcal M \times $$\{ t = 0 \} $},
                \end{array} \right.
\end{equation}
which most efficiently dissipates energy, and the conservative Ginzburg-Landau-Schr\"odinger flow, or Gross-Pitaevskii dynamics
\begin{equation} \label{DGLSintro}
    \left\{\begin{array}{ll}
                 iu_t^{\varepsilon} - \bigtriangleup _{\mathcal M} u^{\varepsilon} = \frac{u^{\varepsilon}}{\varepsilon^2}(1-|u^{\varepsilon}|^2) & \mbox{in $\mathcal M \times \mathbb R_+$} \\
                 u^{\varepsilon} = u_0^{\varepsilon} & \mbox{on $\mathcal M \times $$\{ t = 0 \} $.}
                \end{array} \right.
\end{equation}
The latter arises in studies of superfluidity, Bose-Einstein condensation and nonlinear optics.
Here $\bigtriangleup _{\mathcal M}$ denotes the Laplace-Beltrami operator and
the positive parameter $\e$ corresponds in the full Ginzburg-Landau energy to the reciprocal of the so-called Ginzburg-Landau parameter. In much--though not all--of our
analysis, we will be looking in the asymptotic regime where $\e\ll 1.$ In this regime, it is by now understood (\cite{J1,JS,Sandier,SS}) that energetically reasonable sequences, bounded in energy on the order
of $\ln{(1/\e)}$, can possess at most a finite number of vortices, that is, zeros carrying non-zero degree. Hence, the analysis of these two infinite dimensional flows can be effectively carried out
by tracking the motion of a finite number of points.

In the plane, this program consists of showing that in the limit as $\e\to 0$, the role of the energy $E_\e$ in dictating the dynamics is effectively replaced by
a so-called renormalized energy $W({\bf a},{\bf d})$ dependent on the finite number of vortex locations ${\bf a}=(a_1,a_2,\ldots,a_n)$ and their associated degrees ${\bf d}=(d_1,d_2,\ldots,d_n)$.
 The pivotal role of $W$ was first revealed to great effect in the stationary planar setting with Dirichlet boundary conditions in \cite{BBH}. The asymptotic motion law for planar vortices of the heat flow
 \eqref{DGLintro}, namely
 \begin{equation} \frac{d}{dt} a_i = - \nabla_{a_i} W(\mathbf{a}, \mathbf{d})\quad\mbox{for}\;i=1,2,\ldots,n,\label{aa}\end{equation}
  was first derived in \cite{JS2} and \cite{L2} under a well-prepared initial data assumption up to the first time of vortex collisions, and then was extended more recently in the series of papers \cite{BOS}, \cite{BOS2}, \cite{BOS3}.

  The corresponding system of ODE's governing the asymptotic behavior of  Gross-Pitaevskii vortices in the plane,
 \begin{equation} \frac{d}{dt} a_i =  \nabla^{\perp}_{a_i} W(\mathbf{a}, \mathbf{d})\quad\mbox{for}\;i=1,2,\ldots,n,\label{bb}\end{equation}
 was established in \cite{CJ1,CJ,LX} and later refined in \cite{JSp}.

A primary goal of the present article is to establish the analogs of \eqref{aa} and \eqref{bb} in the manifold setting. For
this, we will appeal to a result of \cite{B}, where the author identifies the renormalized energy on a Riemannian 2-manifold.
Assuming that the manifold is simply-connected, compact and without boundary, one can apply the Uniformization Theorem to assert the existence of a conformal map
$h: {\mathcal {M}} \rightarrow \R^2 \bigcup\{\infty \}$, so that the metric $g$ is given by
\begin{equation} \label{metric}
e^{2f}(dx_1^2 + dx_2^2),
\end{equation}
for some smooth function $f$. Thus one may identify a vortex $a_i\in\mathcal M$ with a point $b_i=h(a_i)\in\R^2 \bigcup \{ \infty \}$. Writing
$\mathbf b = (b_1, b_2, ..., b_n)$ with associated degrees $\mathbf d = (d_1, d_2, ..., d_n)$, a result in \cite{B} identifies  the renormalized energy as
\begin{equation}
    W(\mathbf b, \mathbf d) := \pi \sum_{i=1}^n d_i^2f(b_i) - \pi \sum_{i\neq j} d_i d_j \ln|b_i - b_j|.\label{Wdefna}
\end{equation}

In Section 3 we derive the asymptotic motion law analogous to \eqref{aa} for the heat flow \eqref{DGLintro} on a simply-connected, compact manifold, valid up to the first time of vortex collisions,
under an assumption of well-prepared initial data. We follow the basic scheme laid out in \cite{JS2,L2}. As in the planar case, one expects the motion to be logarithmically slow (cf. \cite{JS2,L2,RS}), so one must first re-scale in time. Then the key steps are to establish an identity which allows one to localize the rate of change of energy about a vortex (Proposition \ref{P})
and to derive a PDE for the energy density itself, Proposition \ref{eflow}. In the present setting this entails a new term involving the Gaussian curvature of $\mathcal{M}$.
The main result of this section is the content of Theorem \ref{Thm2}.

In Section 4 we turn to the vortex law in the dispersive setting on a manifold and generally follow the approach of \cite{CJ1,CJ}. This involves calculating a formula
for the evolution of the Jacobian of a solution of \eqref{DGLSintro}, Proposition \ref{P2}. We then assume that the initial data $u^\varepsilon_0$ is almost energy minimizing, cf. \eqref{aem}, and after a series of results relating the $\e$-limit of the solution to the canonical harmonic map on $\mathcal{M}$, we arrive at our main result,
Theorem \ref{Thm22}, yielding the motion law
\begin{equation}
 d_i \frac{d}{dt} b_i = - \frac{1}{\pi} (\nabla_g^{\perp})_{b_i} W(\mathbf{b}, \mathbf{d})\quad\mbox{for}\;i=1,2,\ldots,n.\label{vM}
\end{equation}
The evolutions \eqref{bb} and \eqref{vM} are known as the point-vortex problem, for the plane and a manifold, respectively. Alternatively,  one can obtain \eqref{bb}
from the Euler equations in the singular limit where vorticity is concentrated at points. This Hamiltonian system has been studied extensively,
and we will not attempt to list all references here, but an excellent source is \cite{N}, where in particular one can find a description of the
state of the art with regard to \eqref{vM} on $S^2$. Regarding the connection between solutions of \eqref{DGLSintro} and \eqref{vM} on $S^2$ we also mention the recent
result \cite{GS} relating rotating solutions of Gross-Pitaevskii to relative equilibria of the point-vortex problem.

In the final two sections of this paper we return to the Ginzburg-landau heat flow and its asymptotic limit. In light of the dissipative nature of this evolution, one expects that
generically, vortices will tend to annihilate each other. This is particularly the expectation on a closed manifold since the total degree of all vortices must be zero, and
for example, there is no Dirichlet condition
as in \cite{BBH} to force the presence of vortices, nor are we considering any applied magnetic field as in \cite{C,CS}.
Vortex annihilation results in the plane for \eqref{DGLintro} were first established in \cite{BCPS} for any finite $\e$ and later the
previously mentioned investigations  \cite{BOS}, \cite{BOS2}, \cite{BOS3} carried this out on bounded planar domains in the regime $\e\ll 1$ under very mild
assumptions on the initial data.
In \cite{KC}, the first author addresses the question of whether there can exist stable vortex configurations
in the sense of second variation for the energy $E_\e$ on a closed manifold without boundary
when $\e$ is small and she also
presents an annihilation result valid for any $\e$ for the flow \eqref{DGLintro} augmented with a Dirichlet condition on a manifold with boundary.

In Section 5 we consider the special case $\mathcal{M}=S^2$ and analyze the corresponding system of ODE's
 \begin{equation*} \frac{d}{dt} a_i =  -(\nabla_g)_{a_i} W(\mathbf{a}, \mathbf{d})\quad\mbox{for}\;i=1,2,\ldots,n.\end{equation*}
We offer a natural definition of how to extend this system past the time of a collision of two or more vortices and then working
with this definition, we establish in Theorem \ref{anh} a sufficient condition for annihilation of all vortices in finite time
involving an assumption of initial clustering of all vortices. We also estimate the time it takes for annihilation.

In Section 6, we remain in the setting of $S^2$ but turn to the PDE \eqref{DGLintro} for fixed $\e$, not necessarily small. Here we derive
in Proposition \ref{wee} certain weighted energy identities. These in particular provide evidence that a similar annihilation result for the heat flow
should hold under a clustering assumption analogous to the one from Section 5. From these identities it follows immediately that any critical point
of Ginzburg-Landau on the two-sphere should satisfy moment identities suggesting a balanced placement of vortices, cf. Corollary \ref{symmetry}.

We begin with a section introducing notation and then proceed as outlined above.
\vskip.5in
\noindent
{\bf Acknowledgment.} The research of both authors was generously supported by NSF grant DMS-1101290 and a Simons Foundation Collaboration Grant.
\vskip.5in
\section{Notation}
Let $\mathcal M$ be a 2- manifold equipped with metric $g$ and let T($\mathcal M$) be the tangent bundle of $\mathcal M$.
For $X, Y \in T(\mathcal M)$, $f: {\mathcal M} \rightarrow \mathbb R$, $u: {\mathcal M} \rightarrow {\mathbb R}^2$, and  $v: {\mathcal M} \rightarrow {\mathbb R}^2$, we write $u = (u_1, u_2)$, $v = (v_1, v_2)$ and define the following notation.
\[
    dv_g = \sqrt{|\det g|} dx_1 \wedge dx_2.
\]
\[
    \nabla_g f \mbox{: the gradient of } f \mbox{ on } \mathcal M.
\]
\[
    \nabla_g^\perp f \mbox{: the skew gradient of } f \mbox{ on } \mathcal M.
\]
Note that $\nabla_g f$ and $\nabla_g^\perp f \in T({\mathcal M})$. When ${\mathcal M} = {\mathbb R}^2$, we have $\nabla_g^\perp f = {\mathbb J} \nabla_g f$, where
\[
    {\mathbb J} = \left(\begin{array}{ll}
                 \; 0 & 1 \\
               -1 & 0
                \end{array} \right).
\]
We will also write
\[
    \nabla_g u := (\nabla_g u_1, \nabla_g u_2), \quad \nabla_g^\perp u := (\nabla_g^\perp u_1, \nabla^\perp_g u_2),
\]
\[
    {\nabla_g}_X Y \mbox{: the covariant derivative of } Y \mbox{ in the direction of } X,
\]
\[
    \langle \cdot, \cdot \rangle_g \mbox{: the inner product of tangent vectors on } \mathcal M,
\]
and
\[
    \langle \mbox{Hess } f (X), Y \rangle_g := \langle {\nabla_g}_X \nabla _g f, Y \rangle_g.
\]
When the inner product involves the gradient of vector-valued functions, we define the notation $\langle \cdot, \cdot \rangle_g$ in two cases:
\[
    \langle \nabla_g u, \nabla_g f \rangle_g = \langle \nabla_g f, \nabla_g u \rangle_g := (\langle \nabla_g f, \nabla_g u_1 \rangle_g, \langle \nabla_g f, \nabla_g u_2 \rangle_g) \in {\mathbb R}^2,
\]
\[
     \langle \nabla_g u, \nabla_g v \rangle_g :=  \langle \nabla_g u_1, \nabla_g v_1 \rangle_g +  \langle \nabla_g u_2, \nabla_g v_2 \rangle_g \in \mathbb R.
\]
We will omit the subscript $g$ for ${\mathcal M} = {\mathbb R}^2$.

For $\phi \in \mathbb R$ and $\mathbf A = (A_1, A_2) \in {\mathbb R^2}$, we define
\[
    \mathbf A^\perp = (A_2, -A_1), \quad \mathbf n(\phi) = (\cos \phi, \sin \phi), \quad \mathbf t(\phi) = - \mathbf n^\perp(\phi).
\]
Then for any $x_0 \in {\mathbb R}^2$, we define $\theta(x-x_0)$ to be the angular polar coordinate of $x$ centered at $x_0$ such that when $x \neq x_0$,
\[
    \mathbf n(\theta(x-x_0)) = \frac{x-x_0}{|x-x_0|}.
\]
When $x_0 = 0$, we will simply write $\mathbf n$ for $\mathbf n (\theta(x))$ and $\mathbf t$ for $\mathbf t (\theta(x))$).

\section{Motion Law for Ginzburg-Landau Vortices on a 2-Manifold}
Let $\mathcal M$ be a smooth, simply connected compact surface without boundary, and $g$ be a metric on $\mathcal M$. We consider the initial value problem
\begin{equation} \label{DGL}
                \left\{\begin{array}{ll}
                 u_t^{\varepsilon} - \bigtriangleup _{\mathcal M} u^{\varepsilon} = \frac{u^{\varepsilon}}{\varepsilon^2}(1-|u^{\varepsilon}|^2) & \mbox{in $\mathcal M \times \mathbb R_+$} \\
                 u^{\varepsilon} = u_0^{\varepsilon} & \mbox{on $\mathcal M \times $$\{ t = 0 \} $}.
                \end{array} \right.
\end{equation}
Here for convenience we will associate $\mathbb C$ with $\mathbb R^2$ and consider $u:{\mathcal M} \times \mathbb R_+ \rightarrow \mathbb R^2$. Note that \eqref{DGL} is the heat flow of the Ginzburg-Landau energy defined by
\begin{equation} \label{energy}
    E_{\varepsilon}(u) = \int_{\mathcal M} \frac{|\nabla_g u|^2}{2} + \frac{(1-|u|^2)^2}{4 \varepsilon^2} dv_g :=  \int_{\mathcal M} e_{\varepsilon} (u) dv_g.
\end{equation}
As mentioned in the introduction, from the Uniformization Theorem, there is a conformal map $h: \mathcal M \rightarrow \mathbb R$$^2 \bigcup\{\infty \}$, so that the metric $g$ is given by
\begin{equation} \label{metric1}
e^{2f}(dx_1^2 + dx_2^2),
\end{equation}
for some smooth function $f$. Thus we may identify points in $\mathcal M$ with points in ${\mathbb R}^2 \bigcup \{ \infty \}$. For $\mathbf b = (b_1, b_2, ..., b_n)$ and $\mathbf d = (d_1, d_2, ..., d_n)$  such that $b_i \in {\mathbb R}^2$, $d_i \in {\mathbb Z} \setminus \{0\}$ for all $i$, the renormalized energy can be written as \cite{B}
\begin{equation}
    W(\mathbf b, \mathbf d) := \pi \sum_{i=1}^n d_i^2f(b_i) - \pi \sum_{i\neq j} d_i d_j \ln|b_i - b_j|.
\end{equation}
To study the dynamics of vortices, we rescale the time variable by a factor $|\ln \varepsilon|$ and set $v^{\varepsilon}(p,t) = u^{\varepsilon}(p, |\ln \varepsilon| t)$. Then $v^{\varepsilon}$ solves
\begin{equation} \label{RTDGL}
                \left\{\begin{array}{ll}
                 \frac{1}{|\ln \varepsilon|}v_t^{\varepsilon} - \bigtriangleup _{\mathcal M} v^{\varepsilon} = \frac{v^{\varepsilon}}{\varepsilon^2}(1-|v^{\varepsilon}|^2) & \mbox{in $\mathcal M \times \mathbb R_+$} \\
                 v^{\varepsilon} = u_0^{\varepsilon} & \mbox{on $\mathcal M \times $$\{ t = 0 \} $}.
                \end{array} \right.
\end{equation}
From the heat flow structure, we easily establish that for $v^\varepsilon$ solving \eqref{RTDGL}, one has
\begin{equation} \label{Ed}
    \frac{d}{dt} \int_{\mathcal M} e_{\varepsilon}(v^{\varepsilon}) dv_g \leq 0,
\end{equation}
which leads to the global existence of a solution to \eqref{RTDGL} for smooth initial data $u^\varepsilon_0$. A key tool in capturing the motion of vortices is the following:
\begin{prop} \label{P}
Let $v^\varepsilon$ be a solution to \eqref{RTDGL} and let $\eta : \mathcal{M} \rightarrow \mathbb{R}$ be any smooth function. Then we have
\begin{align} \label{Peq}
\frac{d}{dt} \int_{\mathcal M} \eta e_{\varepsilon}(v^{\varepsilon}) dv_g = & |\ln \varepsilon| \int_{\mathcal M} \left[ \langle (\mbox{{\rm Hess} } \eta)(\nabla_g  v^{\varepsilon}), \nabla_g v^{\varepsilon}\rangle_g  - \bigtriangleup_{\mathcal M} \eta e_{\varepsilon}(v^{\varepsilon}) \right] dv_g\notag \\
    &-\frac{1}{|\ln \varepsilon|} \int_{\mathcal M} \eta |v_t^{\varepsilon}|^2 dv_g.
\end{align}
\end{prop}
\begin{proof}
 Taking the scalar product of both sides of \eqref{RTDGL} with $\eta v_t^{\varepsilon}$ and integrating we have
\begin{equation} \label{p11}
    \frac{1}{|\ln \varepsilon|} \int_{\mathcal M} \eta |v_t^{\varepsilon}|^2 dv_g - \int_{\mathcal M} \eta v_t^{\varepsilon} \cdot \bigtriangleup_{\mathcal M} v^{\varepsilon} dv_g = -\frac{d}{dt} \int_{\mathcal M} \eta \frac{V(v^{\varepsilon})}{\varepsilon^2},
\end{equation}
where $V(u) := \frac{(1-|u|^2)^2}{4}$. Integrating by parts for the second term we obtain
\begin{equation} \label{p12}
    \int_{\mathcal M} \eta v_t^{\varepsilon} \cdot \bigtriangleup_{\mathcal M} v^{\varepsilon} dv_g = -\int_{\mathcal M} v_t^{\varepsilon} \cdot \langle \nabla_g \eta, \nabla_g v^{\varepsilon} \rangle_g dv_g - \frac{d}{dt} \int_{\mathcal M} \eta \frac{|\nabla_g v^{\varepsilon}|^2}{2} dv_g.
\end{equation}
Combining \eqref{p11} and \eqref{p12} gives
\begin{equation}\label {p13}
    \frac{d}{dt} \int_{\mathcal M} \eta e_{\varepsilon}(v^{\varepsilon}) dv_g = -\int_{\mathcal M} v_t^{\varepsilon} \cdot \langle \nabla_g \eta, \nabla_g v^{\varepsilon} \rangle_g dv_g -  \frac{1}{|\ln \varepsilon|} \int_{\mathcal M} \eta |v_t^{\varepsilon}|^2 dv_g.
\end{equation}
To calculate the first term of  \eqref{p13}, we use \eqref{RTDGL} again. Then
\begin{align} \label{p14}
    \int_{\mathcal M} v_t^{\varepsilon} \cdot & \langle \nabla_g \eta, \nabla_g v^{\varepsilon} \rangle_g dv_g \notag \\
    &= \frac{1}{|\ln \varepsilon|} \int_{\mathcal M} \left[ \bigtriangleup_{\mathcal M} \cdot \langle \nabla_g \eta, \nabla_g v^{\varepsilon} \rangle_g  -  \langle \nabla_g \eta, \nabla_g \frac{V(v^{\varepsilon} )}{\varepsilon^2}\rangle_g \right] dv_g \notag \\
    &= \frac{1}{|\ln \varepsilon|} \int_{\mathcal M} \left[ \bigtriangleup_{\mathcal M} v^{\varepsilon} \cdot \langle \nabla_g \eta, \nabla_g v^{\varepsilon} \rangle_g + \frac{V(v^{\varepsilon})}{\varepsilon^2} \bigtriangleup_{\mathcal M} \eta \right] dv_g.
\end{align}
Through integration by parts and the geometric identity (\cite{P}, p. 207)
\[
    \langle \nabla_g \eta, {\nabla_g}_{\nabla_g v^\varepsilon} \nabla_g v^\varepsilon\rangle_g =  \langle \nabla_g v^\varepsilon, {\nabla_g}_{\nabla_g \eta} \nabla_g v^\varepsilon\rangle_g,
\]
we have
\begin{align} \label{p15}
    \int_{\mathcal M} \bigtriangleup_{\mathcal M} v^{\varepsilon} & \cdot \langle \nabla_g \eta, \nabla_g v^{\varepsilon} \rangle_g dv_g = -\int_{\mathcal M} \langle \nabla_g  v^{\varepsilon}, \nabla_g \langle \nabla_g \eta, \nabla_g  v^{\varepsilon} \rangle_g \rangle_g \notag \\
    &= -\int_{\mathcal M} \left[ \frac{1}{2} \langle \nabla_g |\nabla_g v^{\varepsilon}|^2, \nabla_g \eta \rangle_g + \langle (\mbox{Hess } \eta)(\nabla_g v^{\varepsilon}), \nabla_g v^{\varepsilon} \rangle_g \right] dv_g \notag \\
    &= \int_{\mathcal M} \left[ \frac{|\nabla_g v^{\varepsilon}|^2}{2} \bigtriangleup_{\mathcal M} \eta - \langle (\mbox{Hess } \eta)(\nabla_g v^{\varepsilon}), \nabla_g v^{\varepsilon} \rangle_g \right] dv_g
\end{align}
Therefore by \eqref{p13}, \eqref{p14}, and \eqref{p15} we derive \eqref{Peq}.
\end{proof}

At this point we assume that the initial data is well-prepared in the following sense.
We assume there are precisely $n$ zeros $\{\alpha^{\varepsilon}_1, \alpha^{\varepsilon}_2, ..., \alpha^{\varepsilon}_n\} \subset \mathcal M$ of $u^\varepsilon_0$ such that
\begin{equation} \label{c1}
    R := \frac{1}{3} \min_{0<\varepsilon \leq 1} \min_{i \neq j} dist(\alpha^{\varepsilon}_i , \alpha^{\varepsilon}_j)> 0,
\end{equation}
where $dist(p,q)$ refers to geodesic distance between $p$ and $q$ on $\mathcal M$ and
\begin{equation} \label{c2}
    d_i := \mbox{deg}(u^{\varepsilon}_0, \partial B_R^g(\alpha^{\varepsilon}_i)) \in \{ 1, -1 \} \mbox{ for } 1 \leq i \leq n, \quad  \sum_{i=1}^n d_i = 0.
\end{equation}
We further assume that for some positive constant $c$ independent of $\e$ we have the bounds
\begin{equation} \label{c3}
    \int_{\mathcal M} e_{\varepsilon} (u^{\varepsilon}_0) dv_g \leq n \pi |\ln \varepsilon| + c,
\end{equation}
\begin{equation}
    \sup_{\varepsilon \in (0,1]} \{ e_\varepsilon(u^{\varepsilon}_0(p)) : dist(p,\alpha^\varepsilon_i) \geq \frac{R}{2} \mbox{ for all } i=1, 2, ...,n\} \leq c,
\end{equation}
\begin{equation}
    \inf_{\varepsilon \in (0,1]} \{ |u^{\varepsilon}_0(p)| : dist(p,\alpha^\varepsilon_i) \geq \frac{R}{2} \mbox{ for all } i=1, 2, ...,n\} \geq \frac{3}{4},
\end{equation}
\begin{equation}
    |u^{\varepsilon}_0| \leq 1, \quad \varepsilon|\nabla_g u^{\varepsilon}_0| + \varepsilon^2|\mbox{Hess } u^{\varepsilon}_0| \leq c.
\end{equation}
We also assume that there exist points $\{\alpha_i\} \subset {\mathcal M}$ such that
\begin{equation} \label{c4}
    \lim_{\varepsilon \rightarrow 0} \alpha^{\varepsilon}_i = \alpha_i \; \mbox{ for } 1 \leq i \leq n,
\end{equation}
and
\begin{equation}
    \frac{e_\varepsilon(u^\varepsilon_0)}{|\ln \varepsilon|} \rightharpoonup \pi \sum_{i=1}^n \delta_{\alpha_i} \mbox{ in the sense of distributions.}
\end{equation}
The last condition in fact comes from the stationary results in \cite{BBH} (Theorem VII.2, Theorem VII.3). Next, to obtain the regularity result, we show that the energy density $e_\varepsilon$ defined in \eqref{energy} satisfies a certain PDE.

\begin{prop}\label{eflow}
Let $u^\varepsilon$ be a solution to \eqref{DGL}. Then the energy density $e_\varepsilon(u^\varepsilon)$ satisfies
\begin{align}
    [e_\varepsilon(u^\varepsilon)]_t - \bigtriangleup_{\mathcal M} [e_\varepsilon(u^\varepsilon)] = & \frac{2}{\varepsilon^2}(1-|u^\varepsilon|^2)|\nabla_g  u^\varepsilon|^2 - |\mbox{{\rm{Hess}} } u^\varepsilon|^2 - \frac{4}{\varepsilon^2} |u^\varepsilon \cdot \nabla_g u^\varepsilon|^2 \notag \\
    & - \frac{1}{\varepsilon ^4}(1-|u^\varepsilon|^2)^2|u^\varepsilon|^2 - |\nabla_g u^\varepsilon|^2 K_{\mathcal M},
\end{align}
where $K_{\mathcal M}$ is the Gaussian curvature of $\mathcal M$.
\end{prop}

\begin{proof}
From the definition
\[
    e_\varepsilon(u^\varepsilon) := \frac{|\nabla_g u^\varepsilon|^2}{2} + \frac{(1-|u^\varepsilon|^2)^2}{4\varepsilon^2},
\]
we have
\begin{equation} \label{p12-1}
     [e_\varepsilon(u^\varepsilon)]_t = \langle \nabla_g u^\varepsilon, \nabla_g u^\varepsilon_t  \rangle_g - \frac{1-|u^\varepsilon|^2}{\varepsilon^2} u^\varepsilon \cdot u^\varepsilon_t,
\end{equation}
and
\begin{align} \label{p12-2}
    \bigtriangleup_{\mathcal M} [e_\varepsilon(u^\varepsilon)] = & \bigtriangleup_{\mathcal M} \frac{|\nabla_g u^\varepsilon|^2}{2} - \frac{1-|u^\varepsilon|^2}{\varepsilon^2}u^\varepsilon \cdot \bigtriangleup_{\mathcal M} u^\varepsilon \notag \\
    & - \frac{1-|u^\varepsilon|^2}{\varepsilon^2} |\nabla_g u^\varepsilon|^2 + \frac{2}{\varepsilon^2} |u^\varepsilon \cdot \nabla_g u^\varepsilon|^2.
\end{align}
Substituting \eqref{DGL} into \eqref{p12-1} we derive
\begin{align}\label{p12-3}
    [e_\varepsilon(u^\varepsilon)]_t = & \langle \nabla_g u^\varepsilon, \nabla_g \left[\bigtriangleup_{\mathcal M} u^\varepsilon + \frac{u^\varepsilon}{\varepsilon^2}(1-|u^\varepsilon|^2)\right]  \rangle_g - \frac{1-|u^\varepsilon|^2}{\varepsilon^2} u^\varepsilon \cdot \bigtriangleup_{\mathcal M} u^\varepsilon \notag \\
    & - \frac{1}{\varepsilon^4}(1-|u^\varepsilon|^2)^2|u^\varepsilon|^2 \notag \\
    = & \langle \nabla_g u^\varepsilon, \nabla_g \bigtriangleup_{\mathcal M} u^\varepsilon \rangle_g + \frac{1-|u^\varepsilon|^2}{\varepsilon^2} |\nabla_g u^\varepsilon|^2 - \frac{2}{\varepsilon^2} |u^\varepsilon \cdot \nabla_g u^\varepsilon|^2 \notag \\
    & - \frac{1-|u^\varepsilon|^2}{\varepsilon^2} u^\varepsilon \cdot \bigtriangleup_{\mathcal M} u^\varepsilon - \frac{1}{\varepsilon^4}(1-|u^\varepsilon|^2)^2|u^\varepsilon|^2.
\end{align}
Using the Bochner-Weitzenb\"{o}ck  formula (see \cite{P}, Chapter 7, Proposition 33):
\[
    \bigtriangleup_{\mathcal M} \frac{|\nabla_g u^\varepsilon|^2}{2} = |\mbox{Hess } u^\varepsilon|^2 + \langle \nabla_g u^\varepsilon, \nabla_g \bigtriangleup_{\mathcal M} u^\varepsilon \rangle_g + |\nabla_g u^\varepsilon|^2 K_{\mathcal M},
\]
and combining \eqref{p12-2} and \eqref{p12-3}, we have the desired identity.
\end{proof}

Since $\mathcal M$ is smooth and compact, $K_{\mathcal M}$ is bounded. Hence there exists $\varepsilon_{\mathcal M}>0$ such that the proof of the main regularity result from \cite{JS2} can be readily adapted to the setting on a manifold if we set $0< \varepsilon < \varepsilon_{\mathcal M}$:

\begin{lem}[Regularity, cf. \cite{JS2}, Theorem 6.1] \label{reg}
    Let $B^g_r(q)$ be a geodesic ball in $\mathcal M$ with center at $q \in \mathcal M$ and radius $r$. Let $0 < \varepsilon < \min\{1, r, \varepsilon_{\mathcal M}\}$. Suppose that $u^\varepsilon$ is a solution to \eqref{DGL} in $B^g_{2r}(q) \times (0,4r^2)$ such that
\[
    \sup\{ \int_{B^g_{2r}(q)} e_\varepsilon (u^\varepsilon(\cdot, t)) dv_g: t \in [0,4r^2] \} \leq k_1.
\]
Then there is a constant $C=C(k_1, \mathcal M)$ such that
\[
     e_\varepsilon (u^\varepsilon(p,t)) \leq \frac{C}{r^2} \quad \mbox{for } (p,t) \in B^g_{r}(q) \times [r^2,4r^2].
\]
If in addition
\[
     e_\varepsilon (u^\varepsilon(p,0)) \leq k_1 \quad \mbox{for } p \in B^g_{2r}(q),
\]
then
\[
    e_\varepsilon (u^\varepsilon(p,t)) \leq \frac{C}{r^2} \quad \mbox{for } (p,t) \in B^g_{r}(q) \times [0,4r^2].
\]
\end{lem}

With Lemma \ref{reg} we may extend the following result stated in \cite{JS} to our setting on a manifold:
\begin{thm}[cf. \cite{JS}, Lemma 5.1, Proposition 5.5] \label{Thm1}
    Suppose $v^\varepsilon$ is a solution to \eqref{RTDGL} in ${\mathcal M} \times [0,T_0]$. Then there exists a subsequence $\{\varepsilon_n\}$ and a finite set of points $\{a_i(t)\}_{i=1}^n \subset \mathcal M$ such that $a_i(0) = \alpha_i,$ and for $R$ defined in \eqref{c1}, we have

    \begin{equation} \label{lowerbd}
        \int_{B^g_R(a_i)} e_{\varepsilon} (v^{\varepsilon}(\cdot, t)) dv_g \geq \pi \ln \frac{R}{\varepsilon} - c \quad \mbox{for all } 1 \leq i \leq n, \; \varepsilon \in (0,1).
    \end{equation}
Moreover, as $\varepsilon_n \rightarrow 0$,
    \[
        \frac{e_{\varepsilon_n}(v^{\varepsilon_n})}{|\ln \varepsilon_n|} \rightharpoonup \pi \sum_{i=1}^n \delta_{a_i}
    \]
    in the sense of distributions for $t \in [0, T_1)$, where
    \[
        T_1 := \inf \{ t>0: \min_{i \neq j} \{\mbox{dist}(a_i(t), a_j(t))\}=0 \}.
    \]
    Let $b_i: [0, T_1] \rightarrow {\mathbb R}^2$ be defined as $b_i(t) = h(a_i(t))$, and let $\Theta = \sum_1^n d_i \theta_i$, where $\theta_i(t) = \theta (x - b_i(t))$. Then, after
    perhaps an adjustment by a constant rotational factor $e^{i\theta_0}$, we have
    \begin{equation} \label{thm1eq1}
        v^{\varepsilon_n} \rightarrow v^* = \mathbf{n} (\Theta)
    \end{equation}
    uniformly on any compact subset of $\Omega := \{ (x,t) \in \mathbb{R}^2 \times [0,T_1): x \neq b_i(t) \mbox{ for all } i \}$.
    Furthermore,
    \begin{equation} \label{thm1eq2}
        |\nabla v^{\varepsilon_n}|^2 \mbox{ and } 2e_{\varepsilon_n} (v^{\varepsilon_n}) \rightarrow |\nabla v^*|^2 \quad \mbox{in } L^1_{loc}(\Omega).
    \end{equation}
\end{thm}

Since the method of proving Theorem \ref{Thm1} used in \cite{JS} is independent of the geometry, we can apply the same proof here with only minor changes. Indeed, for a compact manifold $\mathcal M$ without boundary, the proof of the convergence is even easier since we do not have to deal with the Dirichlet boundary condition. This means, in particular, we do not have to adjust the phase of the limiting map $v^*$ for specified boundary behavior. We now come to the main result.

\begin{thm} \label{Thm2}
    The functions $\{b_i(t)\}$ given in Theorem \ref{Thm1} are differentiable, and they satisfy the system of ODE's:
    \begin{equation} \label{thmode}
        \left\{\begin{array}{ll}
        \frac{d}{dt} b_i = - \frac{1}{\pi} (\nabla_g)_{b_i} W(\mathbf{b}, \mathbf{d}) & \mbox{for } t \in (0,T_1) \\
        b_i(0) = \beta_i & \mbox{for } i=1,2, ..., n,
        \end{array} \right.
    \end{equation}
where $\beta_i = h(\alpha_i)$, $\mathbf b = (b_1, b_2, ..., b_n)$, and $\mathbf d = (d_1, d_2, ..., d_n)$ such that each $d_i$ is given by \eqref{c2}.
\end{thm}

\begin{proof}
Fix $t \in [0, T_1)$. Without loss of generality, we may set $i=1$ and assume that $b_1(t)=0$.
For any $\mathbf{A} \in \mathbb{R}^2$ and $\rho >0$, consider a smooth function $\eta$ compactly supported in $B_{2\rho}$ such that $\eta = \langle \mathbf {A}, x \rangle$ in $B_{\rho}$. Then for $0 < \delta < \rho$ small enough, we have
\[
    \eta(b_1(t+\delta)) - \eta(0) = \langle b_1(t+\delta) - b_1(t), \mathbf{A} \rangle.
\]
Now if we integrate \eqref{Peq} from $t$ to $t+\delta$ and divide by $|\ln \varepsilon|$, then Theorem \ref{Thm1} implies
\begin{align}
    \eta(b_1(t+& \delta)) - \eta(0) \notag \\
    = &\frac{1}{\pi} \lim_{\varepsilon \rightarrow 0} \left \{  \int_t^{t+\delta} \int_{\mathcal M} \left[ \langle (\mbox{Hess } \eta)(\nabla_g  v^{\varepsilon}), \nabla_g v^{\varepsilon}\rangle_g  - \bigtriangleup_{\mathcal M} \eta e_{\varepsilon}(v^{\varepsilon}) \right] dv_g d\tau \right. \notag \\
    & \left. - \frac{1}{|\ln \varepsilon|^2} \int_t^{t+\delta} \int_{\mathcal M} \eta |v_t^{\varepsilon}|^2 dv_g d\tau \right\}.
\end{align}
As a consequence of \eqref{Peq} with $\eta=1$, the upper bound for $E_{\e}(v^{\e}_0)$ provided by \eqref{c3} and the lower bound for $E_{\e}(v^{\e}(\cdot,t))$ coming from \eqref{lowerbd},
the second integral on the right-hand side approaches zero. Since $\eta$ is linear in a neighborhood of $b_1$ and compactly supported away from $b_i$ for all $i \neq 1$, the support of Hess $\eta$ does not contain $\{b_1(\tau), b_2(\tau), ..., b_n(\tau) \}$ for $\tau \in [t, t+\delta]$ for small $\delta$. Hence by \eqref{thm1eq1} and \eqref{thm1eq2} we have
\begin{align} \label{9}
    \eta(b_1(t+& \delta)) - \eta(0) \notag \\
    = &\frac{1}{\pi} \int_t^{t+\delta} \int_{\mathcal M} \left[ \langle (\mbox{Hess } \eta)(\nabla_g  v^*), \nabla_g v^*\rangle_g  - \bigtriangleup_{\mathcal M} \eta \frac{|\nabla_g v^*|^2}{2}\right] dv_g d\tau.
\end{align}
Note that for any $\eta: {\mathcal M} \rightarrow {\mathbb R}$ and $v: {\mathcal M} \rightarrow {\mathbb R}^2$, using the metric \eqref{metric} we calculate
\begin{equation} \label{Qs}
    \langle (\mbox{Hess } \eta)(\nabla_g  v), \nabla_g v \rangle_g - \bigtriangleup_{\mathcal M} \eta \frac{|\nabla_g v|^2}{2} = e^{-4f}(q_1 - 2q_2 + q_3),
\end{equation}
where
\begin{equation}\label{Q1}
    q_1 = \langle (\mbox{D}^2 \eta) (\nabla v), \nabla v \rangle - \bigtriangleup \eta \frac{|\nabla v|^2}{2},
\end{equation}
\begin{equation}\label{Q2}
    q_2 = \langle \nabla v, \nabla f \rangle \cdot \langle \nabla v, \nabla \eta \rangle ,
\end{equation}
\begin{equation}\label{Q3}
    q_3 = |\nabla v|^2 \langle \nabla f, \nabla \eta \rangle.
\end{equation}
Now combining \eqref{Qs}-\eqref{Q3} we obtain
\begin{align} \label{10}
\int_{\mathcal M} & \left[ \langle (\mbox{Hess } \eta)(\nabla_g  v^*), \nabla_g v^*\rangle_g  - \bigtriangleup_{\mathcal M} \eta \frac{|\nabla_g v^*|^2}{2}\right] dv_g \notag \\
    = & \int_{B_{2\rho} \setminus B_{\rho}} e^{-2f} \left[ \langle (\mbox{D}^2 \eta) (\nabla v^*), \nabla v^* \rangle - \bigtriangleup \eta \frac{|\nabla v^*|^2}{2}\right] dx \notag \\
    & + \int_{B_{2\rho} \setminus B_{\rho}} e^{-2f} \left[ |\nabla v^*|^2 \langle \nabla f, \nabla \eta \rangle -2 \langle \nabla v^*, \nabla f \rangle \cdot \langle \nabla v^*, \nabla \eta \rangle \right] dx.
\end{align}
Let $\mathbf n = (n_1, n_2)$ be the outward unit normal. Integrating by parts we derive
\begin{align} \label{11}
\int_{B_{2\rho} \setminus B_{\rho}} & e^{-2f} \left[ \langle (\mbox{D}^2 \eta) (\nabla v^*), \nabla v^* \rangle - \bigtriangleup \eta \frac{|\nabla v^*|^2}{2}\right] dx \notag \\
    = & \int_{B_{2\rho} \setminus B_{\rho}} e^{-2f} \left[ \frac{1}{2}(\eta_{x_1 x_1} - \eta_{x_2 x_2})(|v^*_{x_1}|^2 - |v^*_{x_2}|^2) +2 (v^*_{x_1} \cdot v^*_{x_2}) \eta_{x_1 x_2}\right] dx \notag \\
    = & - \int_{B_{2\rho} \setminus B_{\rho}} e^{-2f} \left[ \bigtriangleup v^* \cdot (\eta_{x_1}v^*_{x_1} + \eta_{x_2}v^*_{x_2} )\right] dx \notag \\
       & + \int_{B_{2\rho} \setminus B_{\rho}} e^{-2f} \left[ (f_{x_1} \eta_{x_1} - f_{x_2} \eta_{x_2} )(|v^*_{x_1}|^2 - |v^*_{x_2}|^2) + 2(f_{x_1} \eta_{x_2} + f_{x_2} \eta_{x_1}) v^*_{x_1} \cdot v^*_{x_2} \right] dx \notag \\
       & - \int_{\partial B_\rho} e^{-2f} \left[ \frac{1}{2} (\eta_{x_1}n_1 - \eta_{x_2}n_2)(|v^*_{x_1}|^2 - |v^*_{x_2}|^2) + ( \eta_{x_2}n_1 + \eta_{x_1}n_2) v^*_{x_1} \cdot v^*_{x_2} \right]  ds \notag \\
   = & \int_{B_{2\rho} \setminus B_{\rho}} e^{-2f} \left[ -|\nabla v^*|^2 \langle \nabla f, \nabla \eta \rangle + 2 \langle \nabla v^* , \nabla f \rangle \cdot \langle \nabla v^* , \nabla \eta \rangle \right] dx \notag \\
       & + \int_{\partial B_\rho} e^{-2f} \left[ \frac{1}{2} |\nabla v^*|^2 \langle \nabla \eta, \mathbf n \rangle  - \langle \nabla v^*, \mathbf n \rangle \cdot \langle \nabla v^*, \nabla \eta \rangle\right] ds.
\end{align}
The last equality comes from the fact that $v^*$ is harmonic in $B_{2\rho} \setminus B_{\rho}$. Then by \eqref{10} and \eqref{11} we have
\begin{align} \label{12}
    \int_{\mathcal M} & \left[ \langle (\mbox{Hess } \eta)(\nabla_g  v^*), \nabla_g v^*\rangle_g  - \bigtriangleup_{\mathcal M} \eta \frac{|\nabla_g v^*|^2}{2}\right] dv_g \notag \\
    &= \int_{\partial B_\rho} e^{-2f} \left[ \frac{1}{2} |\nabla v^*|^2 \langle \nabla \eta, \mathbf n \rangle  - \langle \nabla v^*, \mathbf n \rangle \cdot \langle \nabla v^*, \nabla \eta \rangle\right] ds \notag \\
    &:= I_1(\tau) + I_2(\tau).
\end{align}
Recalling from Theorem \ref{Thm1} that $v^* = \mathbf n (\Theta)$, we evaluate the following terms at $x=\rho\mathbf n$:
\begin{align} \label{v2}
    |\nabla v^*|^2 \mathbf n &= \left[ \frac{1}{\rho^2} + \frac{2d_1}{\rho} \langle \sum_{i=2}^n d_i \frac{\mathbf t (\theta_i)}{|x-b_i|}, \mathbf t \rangle \right] \mathbf n + \mathbf E_1(x,\tau) \notag \\
    & = \left[ \frac{1}{\rho^2} + \frac{2d_1}{\rho} \langle \mathbf C(x,\tau), \mathbf t \rangle \right] \mathbf n + \mathbf E_1(x,\tau),
\end{align}
\begin{align} \label{gradv}
    \langle \nabla v^*, \mathbf n \rangle \cdot \nabla v^* = \frac{d_1}{\rho} \langle \mathbf C(x,\tau), \mathbf n \rangle \mathbf t +\mathbf E_2(x, \tau),
\end{align}
where
\begin{equation}
    \mathbf C(x,\tau) := \sum_{i=2}^n d_i \frac{\mathbf t (\theta_i)}{|x-b_i|},
\end{equation}
and $\mathbf E_1$, $\mathbf E_2$ are bounded. Next we expand $e^{-2f}$ near $x=0$,
\begin{align} \label{ef}
    e^{-2f} = e^{-2f(0)} - 2e^{-2f(0)} \langle \nabla f(0), \rho \mathbf n \rangle + O(\rho^2).
\end{align}
Since the most singular part, $e^{-2f(0)}\frac{1}{\rho^2} \mathbf n$, from \eqref{v2} makes no contribution to the integral in \eqref{12}, substituting \eqref{v2}, \eqref{gradv}, and \eqref{ef} into $I_1(\tau)$ and $I_2(\tau)$ we obtain
\begin{align}
    I_1(\tau) = e^{-2f(0)}\int_0^{2\pi} \left[ d_1 \langle \mathbf C(x,\tau), \mathbf t \rangle - \langle \nabla f(0), \mathbf n \rangle \right] \langle \mathbf n, \nabla \eta \rangle d\theta + O_\tau(\rho),
\end{align}
and
\begin{align}
     I_2(\tau) = -e^{-2f(0)}\int_0^{2\pi} d_1 \langle \mathbf C(x,\tau), \mathbf n \rangle \langle \mathbf t, \nabla \eta \rangle d \theta +  O_\tau(\rho),
\end{align}
where $O_\tau(\rho)$ indicates a quantity which is $O(\rho)$ with the implicit constant depending only on $\tau$.
Observe that for any fixed $\mathbf V \in {\mathbb R}^2$ we have
\[
    \int_0^{2\pi} \langle \mathbf V, \mathbf t \rangle \mathbf n d\theta = \pi \mathbf V^\perp = -\int_0^{2\pi} \langle \mathbf V, \mathbf n \rangle \mathbf t d\theta,
\]
and
\[
    \int_0^{2\pi} \langle \mathbf V, \mathbf n \rangle \mathbf n d\theta = \pi \mathbf V.
\]
Thus
\begin{align}\label{Is}
I_1(\tau) + I_2 (\tau) & = - \pi e^{-2f(0)} \langle \nabla f(0) -2 d_1 \mathbf C^\perp(0,\tau), \nabla \eta(0) \rangle +  O_\tau(\rho) \notag \\
    &= - \pi e^{-2f(0)} \langle \nabla f(0) -2 d_1 \mathbf C^\perp(0,\tau), \mathbf A \rangle +  O_\tau(\rho).
\end{align}
Combining \eqref{9}, \eqref{12}, and \eqref{Is} we derive
\begin{align} \label{dbdt}
    \langle \frac{d b_1}{dt} (t), \mathbf A \rangle &= \langle \lim_{\delta \rightarrow 0} \frac{b_1(t+\delta) - b_1(t)}{\delta}, \mathbf A \rangle \notag \\
    &= \langle \lim_{\delta \rightarrow 0}  - e^{-2f(0)} \frac{1}{\delta}\int_t^{t+\delta}  [\nabla f(0) -2 d_1 \mathbf C^\perp(0,\tau) +  O_\tau(\rho) ] d\tau, \mathbf A \rangle \notag \\
    &= \langle - e^{-2f(b_1(t))} [ \nabla f(b_1(t)) -2 d_1 \mathbf C^\perp(b_1(t),t) ], \mathbf A \rangle + O(\rho).
\end{align}
Note that \eqref{dbdt} holds for all $\mathbf A \in {\mathbb R}^2$. Also since
\[
    W(\mathbf b, \mathbf d) := \pi \sum_{i=1}^n f(b_i) - \pi \sum_{i \neq j} d_i d_j \ln |b_i-b_j|,
\]
we have
\begin{align}
    \frac{1}{\pi} (\nabla_g)_{b_1} W(\mathbf b, \mathbf d) &= e^{-2f(b_1)}\left[ \nabla f(b_1) - 2d_1 \sum_{i=2}^n d_i \frac{b_1 - b_i}{|b_1 - b_i|^2} \right] \notag \\
    & = e^{-2f(b_1)}\left[ \nabla f(b_1) - 2d_1 C^\perp(b_1) \right]. \notag
\end{align}
Then taking $\rho \rightarrow 0$ in \eqref{dbdt} we obtain
\begin{equation}
     \frac{d b_1}{dt} = - \frac{1}{\pi} (\nabla_g)_{b_1} W(\mathbf b, \mathbf d).
\end{equation}
\end{proof}

\section{Motion Law for Gross-Pitaevskii Vortices on 2-Manifold}
In this section we consider the initial value problem
\begin{equation} \label{GP}
                \left\{\begin{array}{ll}
                 iu_t^{\varepsilon} - \bigtriangleup _{\mathcal M} u^{\varepsilon} = \frac{u^{\varepsilon}}{\varepsilon^2}(1-|u^{\varepsilon}|^2) & \mbox{in $\mathcal M \times \mathbb R_+$} \\
                 u^{\varepsilon} = u_0^{\varepsilon} & \mbox{on $\mathcal M \times $$\{ t = 0 \} $}.
                \end{array} \right.
\end{equation}
for $u: {\mathcal M} \rightarrow \mathbb C$. Global in time well-posedness in this defocusing setting is provided, for example,
by the results in \cite{Bu}

We define the current
\[
    j(u) := (iu) \cdot \nabla_g u.
\]
Then the relationship $J(u) = \frac{1}{2} \nabla \times j(u)$ between the Jacobian of $u$  and the current on ${\mathbb R}^2$ motivates the following definition of the signed weak Jacobian $J(u)$ for $u \in H^1({\mathcal M})$ via
\[
    \langle J(u), \eta \rangle = \frac{1}{2} \int_{\mathcal M} \langle j(u), \nabla_g^{\perp} \eta \rangle_g d v_g, \mbox{ for all } \eta \in C^1(\mathcal M).
\]
Note that $J(u)$ can be viewed as an element of the dual of $C^1$.
We first establish the following identity, which will be used in our study of dynamics of Gross-Pitaevskii vortices.
\begin{prop} \label{P2}
Let $u^\varepsilon$ be a solution to \eqref{GP} and let $\eta : \mathcal{M} \rightarrow \mathbb{R}$ be any smooth function. Then we have
\begin{equation} \label{P2eq}
\frac{d}{dt} \langle J(u^\varepsilon), \eta \rangle = - \int_{\mathcal M} \langle \nabla_g u^{\varepsilon}, \nabla_{g_{\nabla_g u^{\varepsilon}}} \nabla^{\perp}_g \eta \rangle_g\, dv_g.
\end{equation}
\end{prop}
\begin{proof}
By the definition of $J(u)$, we have
\[
    \frac{d}{dt} \langle J(u^\varepsilon), \eta \rangle = \frac{1}{2} \int_{\mathcal M} \langle \frac{d}{dt} j(u^{\varepsilon}), \nabla_g^{\perp} \eta \rangle_g\, d v_g.
\]
Applying \eqref{GP} we deduce
\begin{align}
    \frac{d}{dt} j(u^{\varepsilon}) =&  (iu^{\varepsilon}_t) \cdot \nabla_g u^{\varepsilon} + (iu^\varepsilon) \cdot \nabla_g u^{\varepsilon}_t \notag \\
    =& (\bigtriangleup_{\mathcal M} u^{\varepsilon} + \frac{1}{\varepsilon^2}(1-|u^{\varepsilon}|^2)u^{\varepsilon}) \cdot \nabla_g u^{\varepsilon} \notag \\
      &- u^{\varepsilon} \cdot \nabla_g (\bigtriangleup_{\mathcal M} u^{\varepsilon} + \frac{1}{\varepsilon^2}(1-|u^{\varepsilon}|^2)u^{\varepsilon}). \notag
\end{align}
Thus
\begin{equation} \label{As}
     \frac{d}{dt} \langle J(u^\varepsilon), \eta \rangle = \frac{1}{2} (A_1 + A_2),
\end{equation}
where
\begin{align}
    A_1 &= \int_{\mathcal M} (\bigtriangleup_{\mathcal M} u^{\varepsilon} + \frac{1}{\varepsilon^2}(1-|u^{\varepsilon}|^2)u^{\varepsilon}) \cdot \langle \nabla_g u^{\varepsilon}, \nabla_g^\perp \eta \rangle_g dv_g \notag \\
    &= \int_{\mathcal M} \left[ \bigtriangleup_{\mathcal M} u^{\varepsilon} \cdot \langle \nabla_g u^{\varepsilon}, \nabla_g^\perp \eta \rangle_g - \langle \nabla_g \frac{V(u^{\varepsilon})}{\varepsilon^2}, \nabla_g^\perp \eta \rangle_g \right] dv_g, \notag
\end{align}
and
\[
    A_2 = -  \int_{\mathcal M} u^{\varepsilon} \cdot \langle \nabla_g (\bigtriangleup_{\mathcal M} u^{\varepsilon} + \frac{1}{\varepsilon^2}(1-|u^{\varepsilon}|^2)u^{\varepsilon}), \nabla_g^\perp \eta \rangle_g dv_g.
\]
Note that $\mbox{div} \nabla_g^\perp \eta =0$. Then integrating by parts we obtain
\begin{align} \label{A1}
    A_1 &=  \int_{\mathcal M} \bigtriangleup_{\mathcal M} u^{\varepsilon} \cdot \langle \nabla_g u^{\varepsilon}, \nabla_g^\perp \eta \rangle_g dv_g \notag \\
    &= - \int_{\mathcal M} \langle \nabla_g u^{\varepsilon}, \nabla_g \langle \nabla_g u^{\varepsilon}, \nabla_g^\perp \eta \rangle_g dv_g \notag \\
    &= - \int_{\mathcal M} \left[ \langle \nabla_g u^{\varepsilon}, \nabla_{g_{\nabla_g u^{\varepsilon}}} \nabla^{\perp}_g \eta \rangle_g + \frac{1}{2} \langle \nabla_g |\nabla_g v^{\varepsilon}|^2, \nabla_g^\perp \eta \rangle_g \right] dv_g \notag \\
    &=  - \int_{\mathcal M} \langle \nabla_g u^{\varepsilon}, \nabla_{g_{\nabla_g u^{\varepsilon}}} \nabla^{\perp}_g \eta \rangle_g dv_g,
\end{align}
and
\begin{align} \label{A2}
    A_2 & = \int_{\mathcal M} (\bigtriangleup_{\mathcal M} u^{\varepsilon} + \frac{1}{\varepsilon^2}(1-|u^{\varepsilon}|^2)u^{\varepsilon}) \cdot \mbox{div}(u^{\varepsilon} \nabla_g^\perp \eta) dv_g \notag \\
    &= \int_{\mathcal M} (\bigtriangleup_{\mathcal M} u^{\varepsilon} + \frac{1}{\varepsilon^2}(1-|u^{\varepsilon}|^2)u^{\varepsilon}) \cdot \langle \nabla_g u^{\varepsilon}, \nabla_g^\perp \eta \rangle_g dv_g \notag \\
    & = A_1.
\end{align}
Combining \eqref{As}, \eqref{A1}, and \eqref{A2} we have the desired equality.
\end{proof}
For $\{u^\varepsilon\}_{0<\varepsilon \leq 1} \subset H^1(\mathcal M)$ satisfying the almost energy minimizing assumption \eqref{am} given below,
 the following upper bound results are proved in \cite{CJ} for sequences $\{ J(u^\varepsilon) \}$ converging weakly as measures.  A similar result is also established in \cite{LX}, Proposition 3.3. Since the proof is independent of the geometry, the results adapt without change to our setting on a manifold:
\begin{lem}[cf. \cite{CJ}, Theorem 1.4.5] \label{limsup}
Let $\{ v^\varepsilon \} \subset H^1 ({\mathcal M})$ be any sequence such that
\[
     J(v^{\varepsilon}) \rightharpoonup \pi \sum_{i=1}^n d_i \delta_{a_i}
\]
weakly as measures. Here $a_i \in \mathcal M$, and $d_i \in \{\pm 1\}$ with $\sum_{i=1}^n d_i =0$.
Suppose that there exists some $\gamma >0$ such that as $\varepsilon \rightarrow 0$,
\begin{equation} \label{am}
     E_{\varepsilon}(v^\varepsilon) \leq n \pi |\ln \varepsilon|  + W(\mathbf b, \mathbf d) + \gamma + o(1),
\end{equation}
where $\mathbf d = (d_1, d_2,..., d_n)$, $\mathbf b = (b_1, b_2, ..., b_n)$, and $b_i = h(a_i)$ (cf. \eqref{metric}-\eqref{Wdefna}). Then there exists a constant $C$ such that for every $\rho >0$,
\begin{equation}
    {\lim \sup}_{\varepsilon \rightarrow 0} ||\frac{j(v^\varepsilon)}{|v^\varepsilon|} - j(H)||^2_{L^2({\mathcal M} \setminus \bigcup_{i=1}^n B^g_\rho(a_i))} \leq C \gamma,
\end{equation}
\begin{equation}
    {\lim \sup}_{\varepsilon \rightarrow 0} ||\nabla_g |v^\varepsilon|||^2_{L^2({\mathcal M} \setminus \bigcup_{i=1}^n B^g_\rho(a_i))} \leq C \gamma.
\end{equation}
Here $H = H(\mathbf a, \mathbf d)$ is the canonical harmonic map, which is unique up to an arbitrary rotation. 
By canonical harmonic map we mean that $H(\mathbf a, \mathbf d)$ is a harmonic map into ${\mathcal S}^1$ with singularities at points $\mathbf a = (a_1, a_2, ..., a_n)$ with $a_i \in \mathcal M$ such that the winding number (degree) of $H$ about $a_i$ is $d_i$ with $\mathbf d = (d_1, d_2, ..., d_n)$.
\end{lem}
Now we make the following assumption on the initial data $u^\varepsilon_0$. We assume that \eqref{c1}, \eqref{c2}, \eqref{c3}, and \eqref{c4} hold. Furthermore, we assume that
\begin{equation} \label{assumH1conv}
    J(u^\varepsilon_0) \rightharpoonup \pi \sum_{i=1}^n d_i \delta_{\alpha_i} \mbox{weakly as measures,}
\end{equation}
where $\alpha_i$ is given by the assumption \eqref{c4}.
Finally we assume that $u^\varepsilon_0$ is almost energy minimizing, i.e. for every $\rho>0$, $u^\varepsilon_0$ satisfies
\begin{equation} \label{aem}
     E_{\varepsilon}(u^\varepsilon_0) \leq n \pi |\ln \varepsilon|  + W(\mathbf \beta, \mathbf d) + o(1)
\end{equation}
as $\varepsilon \rightarrow 0$,
where $\beta = (\beta_1, \beta_2, ..., \beta_n)$, and $\beta_i = h(\alpha_i)$. With these assumptions, the following results established in \cite{CJ} carry over without change to our setting.
\begin{thm}[cf. \cite{CJ}, Theorem 1.4.1]\label{Thm21}
    Suppose $u^{\varepsilon}$ is a solution to \eqref{GP} with the initial data $u^\varepsilon_0$ satisfying the above assumptions. Then there exists a subsequence $\{ \varepsilon_n\}\to 0$  and a finite set of points $\{a_i(t)\}_{i=1}^n \subset \mathcal M$ such that $a_i(0) = \alpha_i$, functions $\{a_i(t)\}_{i=1}^n$ are Lipschitz continuous, and
    \begin{equation} \label{Jconv}
           J(u^{\varepsilon_n}) \rightharpoonup \pi \sum_{i=1}^n d_i \delta_{a_i} \mbox{ weakly as measures}
    \end{equation}
   for $t \in [0, T_1)$, where
    \begin{equation}
        T_1 := \inf \{ t>0: \min_{i \neq j} \{\mbox{dist}(a_i(t), a_j(t))\}=0 \}.\label{T1defn}
    \end{equation}
Moreover,
\begin{equation} \label{convabs}
    |u^{\varepsilon_n}| \rightarrow 1 \mbox{ in } L^2([0,T_1]; L^2 (\mathcal M)),
\end{equation}
and
\begin{equation} \label{jconv}
    j(u^{\varepsilon_n}) \rightharpoonup j(H) \mbox{ in } L^1([0,T_1]; L^1_{loc}  ({\mathcal M} \setminus \bigcup_{i=1}^n a_i )),
\end{equation}
where $H = H(\mathbf a(t), \mathbf d)$ is the canonical harmonic map.
\end{thm}

We can now establish our main result of this section.
\begin{thm} \label{Thm22}
 Let functions $\{a_i (t)\}_{i=1}^n$ be given as in Theorem \ref{Thm21}. Then for $b_i(t) = h(a_i(t))$, the collection $\{b_i(t)\}_{i=1}^n\subset\R^2$ satisfies the system of ODE's:
    \begin{equation} \label{thmode2}
         \left\{\begin{array}{ll}
        d_i \frac{d}{dt} b_i = - \frac{1}{\pi} (\nabla_g^{\perp})_{b_i} W(\mathbf{b}, \mathbf{d}) & \mbox{for } t \in (0,T_1) \\
        b_i(0) = \beta_i,
        \end{array} \right.
    \end{equation}
where $\beta_i = h(\alpha_i)$, $\mathbf b = (b_1, b_2, ..., b_n)$, and $\mathbf d = (d_1, d_2, ..., d_n)$.
\end{thm}
\begin{proof}
\begin{flushleft}
    \bf{-Step 1:}
\end{flushleft}
Let the functions \{$\tilde{b}_i(t)\}$ solve the system of ODE's \eqref{thmode2} let $\tilde{T}_1$ be the time of existence so that
\[
        \tilde{T}_1 := \inf \{ t>0: \min_{i \neq j} \{|\tilde{b}_i(t) - \tilde{b}_j(t)|\}=0 \}.
\]
Define the function $\zeta$ by
\[
    \zeta(t) = \sum_{i=1}^n |d_i b_i(t) - d_i \tilde{b}_i(t)|
\]
for $t \in [0,T_*)$ where $T_*:= \min\{T_1, \tilde{T}_1\}$ and $T_1$ is given by \eqref{T1defn}.
Fix $\sigma >0$. Since $\zeta(0) = 0$, there exists $0< t_{\sigma} < T_*$ such that $\zeta(t) \leq \sigma$ for $t \in [0,t_{\sigma}]$.
Since the functions $b_i(t)$ are Lipschitz continuous, they are differentiable for almost every $t$. Hence, for a.e. $t \in [0,t_{\sigma}]$, we have
\begin{align} \label{dzetadt}
    \frac{d \zeta}{dt} & \leq \sum_{i=1}^n |d_i \frac{db_i}{dt} - d_i \frac{d\tilde{b}_i}{dt}| \notag \\
        & \leq \sum_{i=1}^n |d_i \frac{db_i}{dt} + \frac{1}{\pi} (\nabla_g^{\perp})_{b_i} W(\mathbf{b}, \mathbf{d})| + \frac{1}{\pi} \sum_{i=1}^n | (\nabla_g^{\perp})_{b_i} W(\mathbf{b}, \mathbf{d}) - (\nabla_g^{\perp})_{\tilde{b}_i} W(\tilde{\mathbf{b}}, \mathbf{d})| \notag \\
        & \leq  \sum_{i=1}^n |d_i \frac{db_i}{dt} + \frac{1}{\pi} (\nabla_g^{\perp})_{b_i} W(\mathbf{b}, \mathbf{d})| + C\zeta.
\end{align}
Here we have applied Taylor's Theorem to $\nabla_g^{\perp} W$, valid for $\sigma$ sufficiently small.

Now, fix $t \in [0, t_{\sigma}]$ to be any point of differentiability of $b_i(t)$. Without loss of generality, we may set $i=1$ and assume that $b_1(t)=0$. Pick $\mathbf A \in {\mathbb R}^2$ such that
\begin{equation} \label{term2A}
    |d_1 \frac{db_1}{dt} + \frac{1}{\pi} (\nabla_g^{\perp})_{b_1} W(\mathbf{b}, \mathbf{d})| = \langle  d_1 \frac{db_1}{dt} + \frac{1}{\pi} (\nabla_g^{\perp})_{b_1} W(\mathbf{b}, \mathbf{d}) , \mathbf A \rangle.
\end{equation}
Fix $\rho >0$ and let $\eta$ be a smooth function compactly supported in $B_{2\rho}$ such that $\eta = \langle \mathbf A, x \rangle$ in $B_{\rho}$. Then applying Proposition \ref{P2} and \eqref{Jconv} we deduce
\begin{align} \label{RHS1}
    \langle d_1 \frac{db_1}{dt}(t) , \mathbf A \rangle & = \langle \lim_{\delta \rightarrow 0} \frac{d_1}{\delta} [b_1(t+\delta) - b_1(t)] , \mathbf{A} \rangle \notag \\
    & = - \lim_{\delta \rightarrow 0} \lim_{\varepsilon \rightarrow 0} \frac{1}{\delta \pi} \int_t^{t+ \delta} \int_{\mathcal M} \langle \nabla_g u^\varepsilon, \nabla_{g_{\nabla_g u^\varepsilon}} \nabla^{\perp}_g \eta \rangle_g dv_g d\tau.
\end{align}

\begin{flushleft}
    \bf{-Step 2:}
\end{flushleft}
Let $H$ be given in Theorem \ref{Thm21}. Projecting $H$ onto the plane, we may write $H=u^*:=\mathbf n(\Theta)$, where $\Theta = \sum_{i=1}^n d_i \theta_i$, and $\theta_i(t) = \theta(x-b_i(t))$. We will show that
\begin{equation} \label{lemma32}
    \langle (\nabla_g^{\perp})_{b_1} W(\mathbf{b}, \mathbf{d}) , \mathbf A \rangle = \int_{\mathcal M} \langle \nabla_g u^*, \nabla_{g_{\nabla_g u^*}} \nabla^{\perp}_g \eta \rangle_g dv_g.
\end{equation}
Note that for any smooth $\eta : {\mathcal M} \rightarrow \mathbb R$ and $u : {\mathcal M} \rightarrow {\mathbb R}^2$ in $H^1$, using the metric \eqref{metric1} we calculate
\begin{equation} \label{qs2}
- \langle \nabla_g u, \nabla_{g_{\nabla_g u}} \nabla^{\perp}_g \eta \rangle_g = e^{-4f} (q_1 + 2q_2 -q_3)
\end{equation}
where
\begin{equation} \label{q21}
    q_1 = \langle \nabla u, (D^2 \eta) \mathbb{J} \nabla u \rangle \mbox{ with }
    {\mathbb J} =  \left(\begin{array}{ll}
                 \; 0 & 1 \\
               -1 & 0
                \end{array} \right),
\end{equation}
\begin{equation} \label{q22}
    q_2 = \langle \nabla u, \nabla f \rangle \cdot \langle \nabla u, \nabla^{\perp} \eta \rangle,
\end{equation}
and
\begin{equation} \label{q23}
    q_3 = |\nabla u|^2 \langle \nabla f, \nabla^{\perp} \eta \rangle.
\end{equation}
Taking $0< r< \rho$, and combining \eqref{qs2}-\eqref{q23} we derive
\begin{align} \label{315}
    -\int_{\mathcal M} & \langle \nabla_g u^*, \nabla_{g_{\nabla_g u^*}} \nabla^{\perp}_g \eta \rangle_g dv_g \notag \\
    = & \int_{B_{2r} \setminus B_{r}} e^{-2f} \langle \nabla u^*, (\mbox{D}^2 \eta) {\mathbb J} \nabla u^* \rangle dx \notag \\
    & + \int_{B_{2r} \setminus B_{r}} e^{-2f} \left[ 2 \langle \nabla u^*, \nabla f \rangle \cdot \langle \nabla u^*, \nabla^\perp \eta \rangle - |\nabla u^*|^2 \langle \nabla f, \nabla^\perp \eta \rangle  \right] dx.
\end{align}
Integrating by parts and using the fact that $u^*$ is harmonic in $B_{2r} \setminus B_{r}$ we obtain
\begin{align}
     \int_ {B_{2r} \setminus B_{r}}&  e^{-2f} \langle \nabla u^*, (\mbox{D}^2 \eta) {\mathbb J} \nabla u^* \rangle dx \notag \\
    =& \int_{B_{2r} \setminus B_{r}} e^{-2f} \left[ (|u^*_{x_2}|^2 - |u^*_{x_1}|^2)\eta_{x_1 x_2} + (u^*_{x_1} \cdot u^*_{x_2})(\eta_{x_1 x_1} - \eta_{x_2 x_2})\right] dx \notag \\
    =& \frac{1}{2} \int_{B_{2r} \setminus B_{r}} e^{-2f} \left[ (|u^*_{x_1}|^2 + |u^*_{x_2}|^2)_{x_1} \eta_{x_2} - (|u^*_{x_1}|^2 + |u^*_{x_2}|^2)_{x_2} \eta_{x_1} \right] dx \notag \\
    & + 2 \int_{B_{2r} \setminus B_{r}} e^{-2f} \left[ |u^*_{x_2}|^2 f_{x_2} \eta_{x_1} - |u^*_{x_1}|^2 f_{x_1} \eta_{x_2} + (u^*_{x_1} \cdot u^*_{x_2})(f_{x_1} \eta_{x_1} - f_{x_2} \eta_{x_2}) \right] dx \notag \\
    & + \int_{\partial B_{r}} e^{-2f} \left [ |u^*_{x_1}|^2 \eta_{x_2} n_1 - |u^*_{x_2}|^2 \eta_{x_1} n_2 + (u^*_{x_1} \cdot u^*_{x_2})(\eta_{x_2} n_2 - \eta_{x_1} n_1) \right] ds. \notag
\end{align}
Next, integrating by parts again for the first integral, we have
\begin{align} \label{316}
    \int_ {B_{2r} \setminus B_{r}}&  e^{-2f} \langle \nabla u^*, (\mbox{D}^2 \eta) {\mathbb J} \nabla u^* \rangle dx \notag \\
    = & \int_{B_{2r} \setminus B_{r}} e^{-2f} \left[ (|u^*_{x_2}|^2 - |u^*_{x_1}|^2)(f_{x_2} \eta_{x_1} + f_{x_1} \eta_{x_2}) +2(u^*_{x_1} \cdot u^*_{x_2})(f_{x_1} \eta_{x_1} - f_{x_2} \eta_{x_2}) \right] dx  \notag \\
    & - \int_{\partial B_{r}} e^{-2f} \left[ \frac{1}{2} (|u^*_{x_2}|^2 - |u^*_{x_1}|^2)(\eta_{x_1} n_2 + \eta_{x_2} n_1) + (u^*_{x_1} \cdot u^*_{x_2})(\eta_{x_1} n_1 - \eta_{x_2} n_2) \right] ds \notag \\
    =&  \int_{B_{2r} \setminus B_{r}} e^{-2f} \left[ |\nabla u^*|^2 \langle \nabla f, \nabla^\perp \eta \rangle -2 \langle \nabla u^*, \nabla f \rangle \cdot \langle \nabla u^*, \nabla^\perp \eta \rangle \right] dx \notag \\
    & - \int_{\partial B_{r}} e^{-2f} \left[ \frac{1}{2} |\nabla u^*|^2 \langle \nabla^\perp \eta, \mathbf n \rangle - \langle \nabla u^*, \nabla^\perp \eta \rangle \cdot \langle \nabla u^*, \mathbf n \rangle \right] ds.
\end{align}
Thus from \eqref{315} and \eqref{316}, we have
\begin{align}
    \int_{\mathcal M} & \langle \nabla_g u^*, \nabla_{g_{\nabla_g u^*}} \nabla^{\perp}_g \eta \rangle_g dv_g  \notag \\
    =& \int_{\partial B_{r}} e^{-2f} \left[ \frac{1}{2} |\nabla u^*|^2 \langle \nabla^\perp \eta, \mathbf n \rangle - \langle \nabla u^*, \nabla^\perp \eta \rangle \cdot \langle \nabla u^*, \mathbf n \rangle \right] ds.
\end{align}
Note that $\nabla^\perp \eta = \mathbf A^\perp$ at $x=0$ and that $(\mathbf A^\perp)^\perp = -\mathbf A$. From the type of argument used for calculating \eqref{12}, we obtain
\begin{align}
   \int_{\mathcal M} & \langle \nabla_g u^*, \nabla_{g_{\nabla_g u^*}} \nabla^{\perp}_g \eta \rangle_g dv_g  \notag \\
        & = \langle e^{-2f(b_1(t))} \pi [ \nabla f(b_1(t)) -2 d_1 \mathbf C^\perp(b_1(t),t) ]^\perp, \mathbf A \rangle + O(r) \notag \\
        & = \langle (\nabla_g^{\perp})_{b_1} W(\mathbf{b}, \mathbf{d}) , \mathbf A \rangle + O(r).
\end{align}
Since $r$ can be taken arbitrarily small, we have proved \eqref{lemma32}. Using it we derive
\begin{align} \label{RHS2}
    \langle \frac{1}{\pi} (\nabla_g^{\perp})_{b_1} W(\mathbf{b}(t), \mathbf{d}) , \mathbf A \rangle & = \lim_{\delta \rightarrow 0} \frac{1}{\delta \pi} \int_t^{t+\delta} \langle (\nabla_g^{\perp})_{b_1} W(\mathbf{b}(\tau), \mathbf{d}) , \mathbf A \rangle d\tau \notag \\
    & = \lim_{\delta \rightarrow 0} \frac{1}{\delta \pi} \int_t^{t+\delta} \int_{\mathcal M} \langle \nabla_g u^*, \nabla_{g_{\nabla_g u^*}} \nabla^{\perp}_g \eta \rangle_g dv_g.
\end{align}

\begin{flushleft}
    \bf{-Step 3:}
\end{flushleft}
From \eqref{term2A}, \eqref{RHS1} and \eqref{RHS2},
\begin{align} \label{bddW}
    |d_1 & \frac{db_1}{dt} + \frac{1}{\pi} (\nabla_g^{\perp})_{b_1} W(\mathbf{b}, \mathbf{d})| \notag \\
       = & - \lim_{\delta \rightarrow 0} \lim_{\varepsilon \rightarrow 0} \frac{1}{\delta \pi} \int_t^{t+ \delta} \int_{\mathcal M} \langle \nabla_g u^\varepsilon, \nabla_{g_{\nabla_g u^\varepsilon}} \nabla^{\perp}_g \eta \rangle_g  - \langle \nabla_g u^*, \nabla_{g_{\nabla_g u^*}} \nabla^{\perp}_g \eta \rangle_g dv_g d\tau.
\end{align}
Applying \eqref{qs2}-\eqref{q23} again we find
\begin{align} \label{bddW2}
    \int_{\mathcal M} & \langle \nabla_g u^\varepsilon, \nabla_{g_{\nabla_g u^\varepsilon}} \nabla^{\perp}_g \eta \rangle_g  - \langle \nabla_g u^*, \nabla_{g_{\nabla_g u^*}} \nabla^{\perp}_g \eta \rangle_g dv_g \notag \\
    = & \int_ {B_{2\rho} \setminus B_{\rho}} e^{-4f} [ \langle \nabla u^\varepsilon, (D^2 \eta){\mathbb J} \nabla u^\varepsilon \rangle - \langle \nabla u^*, (D^2 \eta){\mathbb J} \nabla u^* \rangle ] dx \notag \\
        & +  2 \int_ {B_{2\rho} \setminus B_{\rho}} e^{-4f} [ \langle \nabla u^\varepsilon, \nabla f \rangle \cdot \langle \nabla u^\varepsilon, \nabla^\perp \eta \rangle - \langle \nabla u^*, \nabla f \rangle \cdot \langle \nabla u^*, \nabla^\perp \eta \rangle] dx \notag \\
        & -  \int_ {B_{2\rho} \setminus B_{\rho}} e^{-4f} (|\nabla u^\varepsilon|^2 - |\nabla u^*|^2) \langle \nabla f,  \nabla \eta \rangle dx.
\end{align}
Note that in \eqref{bddW2}, each term in the integral is of the form  $F u^\varepsilon_{x_k}  u^\varepsilon_{x_l}$, where $F$ is a smooth function and $k, l \in \{1,2\}$. To control them, we first calculate for $k = 1,2$,
\[
    u^\varepsilon_{x_k} = \frac{j^k(u^\varepsilon)}{|u^\varepsilon|} \frac{iu^\varepsilon}{|u^\varepsilon|} + |u^\varepsilon|_{x_k} \frac{u^\varepsilon}{|u^\varepsilon|},
\]
where we write $j(u^\varepsilon) = (j^1(u^\varepsilon), j^2(u^\varepsilon))$. Thus for $k,l \in \{1,2\}$,
\begin{align}
    \frac{j^k(u^\varepsilon) j^l(u^\varepsilon)}{|u^\varepsilon|^2} & - j^k(u^*)j^l(u^*) \notag \\
        = & \left[\frac{j^k(u^\varepsilon)}{|u^\varepsilon|} -  j^k(u^*) \right] \left[ \frac{j^l(u^\varepsilon)}{|u^\varepsilon|} -  j^l(u^*) \right] \notag \\
        & + j^k(u^*)\left[ \frac{j^l(u^\varepsilon)}{|u^\varepsilon|} -  j^l(u^*) \right] + j^l(u^*)\left[ \frac{j^k(u^\varepsilon)}{|u^\varepsilon|} -  j^k(u^*) \right].
\end{align}
Since $\frac{j(u^\varepsilon)}{|u^\varepsilon|}$ is uniformly bounded in $L^2({\mathcal M} \setminus \bigcup_{i=1}^n B^g_\rho(a_i) \times [t,t + \delta])$, up to a subsequence it converges to some weak limit $\tilde{j}$ in $L^2$. Then from \eqref{convabs} and \eqref{jconv}, we see that
\[
    \tilde{j} = j(u^*),
\]
i.e.
\begin{equation} \label{weaklimj}
    \frac{j(u^\varepsilon)}{|u^\varepsilon|} \rightharpoonup j(u^*) \mbox{ in } L^2({\mathcal M} \setminus \bigcup_{i=1}^n B^g_\rho(a_i) \times [t,t + \delta]).
\end{equation}
Now the almost minimizing energy assumption \eqref{aem} and the conservation of energy $E_\varepsilon$ along the flow \eqref{GP} imply
\begin{align}
    E_{\varepsilon}(u^\varepsilon (\cdot, t)) & = E_{\varepsilon}(u^\varepsilon_0) \notag \\
        & \leq n \pi |\ln \varepsilon| + W(\mathbf \beta ,\mathbf d) + o(1) \notag \\
        & = n \pi  |\ln \varepsilon|  + W(\tilde{\mathbf b}(t) ,\mathbf d) + o(1) \notag \\
        & \leq  n \pi  |\ln \varepsilon|  + W(\mathbf b(t) ,\mathbf d) + C\zeta(t) + o(1).
\end{align}
Here we also used the fact that the renormalized energy $W$ is conserved by the flow $\tilde{\mathbf b}(t)$. Then by \eqref{Jconv} and Lemma \ref{limsup}, we have
\begin{equation} \label{bddj}
    {\lim \sup}_{\varepsilon \rightarrow 0} ||\frac{j(u^\varepsilon)}{|u^\varepsilon|} - j(u^*)||^2_{L^2({\mathcal M} \setminus \bigcup_{i=1}^n B^g_\rho(a_i))} \leq C \zeta(t),
\end{equation}
and
\begin{equation} \label{bdd2}
    {\lim \sup}_{\varepsilon \rightarrow 0} ||\nabla_g |u^\varepsilon|||^2_{L^2({\mathcal M} \setminus \bigcup_{i=1}^n B^g_\rho(a_i))} \leq C \zeta(t).
\end{equation}
Using \eqref{weaklimj} and passing limit $\varepsilon \rightarrow 0$ in \eqref{bddW} we obtain
\[
    |d_1 \frac{db_1}{dt} + \frac{1}{\pi} (\nabla_g^{\perp})_{b_1} W(\mathbf{b}, \mathbf{d})| \leq C \zeta(t).
\]
Since such an estimate holds for each vortex path we combining this with \eqref{dzetadt} to see that for $t \in [0, t_\sigma]$, one has
\[
    \frac{d\zeta}{dt} \leq C \zeta(t)
\]
with $\zeta(0) = 0$. Hence, we have $\zeta \equiv 0$, which implies \eqref{thmode2}.
\end{proof}

\section{Gradient Flow of Point Vortices on ${S}^2$}
In this section we will discuss the limiting vortex motion on ${S}^2 \subset {\mathbb{R}}^3$ given by the system of ODE's \eqref{thmode}. Let $\{ \mathbf{P}_i\}_{i=1}^{2n}$ be $2n$ vortices on ${S}^2$ with degrees $\{ d_i \}_{i=1}^{2n}$ such that $\sum_{i=1}^{2n} d_i = 0$. Here we restrict attention to vortices of initial degrees $\pm 1$. Note that for ${S}^2$, we may write the renormalized energy W as
\begin{equation}
    W = - \pi \sum_{i \neq j} d_i d_j \ln|\mathbf P_i - \mathbf P_j|.\label{rne}
\end{equation}
Defining $T_1 := \inf\{ t>0: \min_{i \neq j} \{|\mathbf{P}_i - \mathbf{P}_j|\} =0 \}$, the vortices $\{ \mathbf{P}_i\}_{i=1}^{2n}$ satisfy the gradient flow:
\begin{equation} \label{odeS2}
        \frac{d \mathbf P_i}{dt} = (\nabla_{{S}^2})_{\mathbf P_i}  d_i \sum_{j \neq i}d_j \ln|\mathbf P_i - \mathbf P_j|_{{\mathbb R}^3} \quad \mbox{for } t \in (0,T_0).
\end{equation}
Now let $\mathbf{p}_i \in {\mathbb{R}}^2$ be the image of $\mathbf{P}_i$ via stereographic projection such that coordinates of $\mathbf{P}_i$ in ${\mathbb{R}}^3$ are given by
\begin{equation}\label{coord}
        \mathbf{P}_i = (\frac{2\mathbf{p}_i}{1+r_i^2},\frac{r_i^2-1}{1+r_i^2}),
\end{equation}
where $r_i^2 = |\mathbf p_i|^2$. In this case the conformal factor $e^{2f(x)} = \frac{4}{(1+|x|^2)^2}$, i.e.
\[
    f(x) = \ln (\frac{2}{1+|x|^2}).
\]
Thus \eqref{thmode} can be written as the following system of ODE's on the plane:
\begin{equation} \label{ode}
        \frac{d \mathbf{p}_i}{dt} \equiv \dot{\mathbf{p}_i} = \frac{(1+r_i^2)^2}{2}
                                           \left( \frac{ \mathbf{p}_i}{1+r_i^2}
                                           + d_i \sum_{j \neq i} d_j \frac{\mathbf{p}_i - \mathbf{p}_j}{|\mathbf{p}_i - \mathbf{p}_j|^2} \right).
\end{equation}
We first establish a useful identity for the quantity $\sum_{i=1}^{2n} \mathbf{P}_i$.
\begin{prop} \label{ExpD}
    Let $\mathbf{V}_0 = \sum_{i=1}^{2n} \mathbf{P}_i$, then $\dot{\mathbf{V}}_0 = \mathbf{V}_0$ in $(0, T_1)$.
\end{prop}

\begin{proof}
Taking the derivative using \eqref{coord} we find
\begin{equation} \label{PDot}
        \dot{\mathbf{V}}_0 = (\sum_{i=1}^{2n} \left[ \frac{2}{1+r_i^2} \dot{\mathbf{p}_i} - \frac{4\mathbf{p}_i \cdot \dot{\mathbf{p}_i}}{(1+r_i^2)^2} \mathbf{p}_i \right] , \sum_{i=1}^{2n} \frac{4\mathbf{p}_i \cdot \dot{\mathbf{p}_i}}{(1+r_i^2)^2}).
\end{equation}
Then using \eqref{ode} we obtain
\begin{align} \label{Pxy}
        \frac{2}{1+ r_i^2} \dot{\mathbf{p}_i} & - \frac{4\mathbf{p}_i \cdot \dot{\mathbf{p}_i}}{(1+r_i^2)^2} \mathbf{p}_i \notag \\
                                             = \mathbf{p} _i & + d_i (1+r_i^2) \sum_{j \neq i} d_j \frac{\mathbf{p}_i - \mathbf{p}_j}{|\mathbf{p}_i - \mathbf{p}_j|^2} - \mathbf{p}_i \left( \frac{2r_i^2}{1+r_i^2} + 2d_i \sum_{j \neq i} d_j \frac{r_i^2 - \mathbf{p}_i \cdot \mathbf{p}_j}{|\mathbf{p}_i - \mathbf{p}_j|^2}\right) \notag \\
                                             = \mathbf{p}_i & \left( 1 - \frac{2r_i^2}{1+r_i^2} - d_i \sum_{j \neq i} d_j \frac{r_i^2 - 2\mathbf{p}_i \cdot \mathbf{p}_j}{|\mathbf{p}_i - \mathbf{p}_j|^2} \right) - \sum_{j \neq i} \frac{d_i d_j  r_i^2 \mathbf{p}_j}{|\mathbf{p}_i - \mathbf{p}_j|^2}  \notag \\
                                                                       &- \sum_{i \neq j}  d_i d_j \frac{\mathbf{p}_i - \mathbf{p}_j}{|\mathbf{p}_i - \mathbf{p}_j|^2}
\end{align}
Since $r_i^2 - 2\mathbf{p}_i \cdot \mathbf{p}_j = |\mathbf{p}_i - \mathbf{p}_j|^2 - r_j^2$ and $d_i \sum_{j \neq i} d_j = -1$, we have
\begin{align} \label{Pxy}
        \frac{2}{1+ r_i^2} & \dot{\mathbf{p}_i} - \frac{4\mathbf{p}_i \cdot \dot{\mathbf{p}_i}}{(1+r_i^2)^2} \mathbf{p}_i \notag \\
                                             & = \mathbf{p}_i \left( 2 - \frac{2r_i^2}{1+r_i^2} \right) - \sum_{j \neq i} \left[ \frac{d_i d_j }{|\mathbf{p}_i - \mathbf{p}_j|^2} (r_i^2 \mathbf{p}_j - r_j^2 \mathbf{p}_i) + d_i d_j \frac{\mathbf{p}_i - \mathbf{p}_j}{|\mathbf{p}_i - \mathbf{p}_j|^2} \right] \notag \\
                                             & =: \frac{2\mathbf{p}_i}{1+r_i^2} - \mathbf{A}_i.
\end{align}
Similarly,
\begin{equation} \label{Pz}
        \frac{4\mathbf{p}_i \cdot \dot{\mathbf{p}_i}}{(1+r_i^2)^2} = \frac{r_i^2-1}{1+r_i^2} + B_i,
\end{equation}
where
\[
        B_i = \sum_{j \neq i} d_i d_j \frac{r_i^2 - r_j^2}{|\mathbf{p}_i - \mathbf{p}_j|^2}.
\]
Note that $\sum_{i=1}^{2n} \mathbf{A}_i = \mathbf{0}$ and $\sum_{i=1}^{2n} B_i =0$. Combining \eqref{PDot}, \eqref{Pxy} and \eqref{Pz}, we derive
\[
        \dot{\mathbf{V}}_0 = \mathbf{V}_0.
\]
\end{proof}

\begin{cor} \label{centerofmass}
If ${\mathbf{V}}_0(0)\not=0$, then
\[
        \frac{d}{dt} \frac{\mathbf{V}_0}{|\mathbf{V}_0|} = 0,
\]
i.e. the projection of $\mathbf V_0$ onto ${S}^2$ is fixed for $t \in [0, T_1)$. We may view this as conservation of the center of mass.
Furthermore, either $\mathbf V_0 (t) \equiv 0$ or else $T_1 < \infty$. Lastly,$\mathbf V_0=0$ is a necessary condition for an equilibrium solution
to \eqref{odeS2}.
 \end{cor}
\noindent{\it Proof.} By Proposition \ref{ExpD}, we have $\mathbf V_0(t) = \mathbf V_0 (0) e^t$ and since $|\mathbf V_0 (t)| \leq 2n$, all of the conclusions
 follow immediately.
\qed

\begin{prop} \label{Decay}
Let  $\mathbf{V}_0$ be as in Proposition \ref{ExpD}. Then $\sum_{i>j} |\mathbf{P}_i-\mathbf{P}_j|^2$ decays in $[0, T_1)$. Furthermore, the decay is strict if $\mathbf V_0 (0) \neq 0$.
\end{prop}

\begin{proof}
Since $|\mathbf{P}_i| = 1$ for all $i$, we have
\begin{align}
\frac{d}{dt} \sum_{i>j} |\mathbf{P}_i-\mathbf{P}_j|^2 &= -2 \frac{d}{dt} \sum_{i>j} \mathbf{P}_i \cdot \mathbf{P}_j \notag\\
                                                                                          &= - \frac{d}{dt} (\mathbf{V}_0 \cdot \mathbf{V}_0) \notag \\
                                                                                          &= -2 \mathbf{V}_0 \cdot \dot{\mathbf{V}}_0. \notag
\end{align}
From Proposition \ref{ExpD} we then have
\[
\frac{d}{dt} \sum_{i>j} |\mathbf{P}_i-\mathbf{P}_j|^2 = -2|\mathbf{V_0}|^2 = -2|\mathbf V_0 (0)|^2 e^{2t}.
\]
\end{proof}

If $T_1 < \infty$, we know that $\mathbf{P}_{i_*} = \mathbf{P}_{j_*}$ at $t=T_1$ for some $i_* \neq  j_*$. Next we will prove that any collision must involve at least two vortices with different signs. In general, we have
\begin{prop}\label{43}
Assume that at the first collision time $t=T_1 < \infty$, the total degree of the colliding vortices is not zero. Denote the absolute value of this total by $l$. Then the collision cannot involve only vortices having degree of one sign. Furthermore, if the collision involves $k$ vortices of one sign and $k+l$ vortices of the opposite sign with $l \geq 1$, then the following inequality gives an upper bound on $l$:
\[
    C_2^k + C_2^{k+l}- k(k+l) < 1,
\]
where $C^k_n := \frac{k!}{n!(k-n)!}.$
\end{prop}
\begin{proof}
Without loss of generality, we assume a collision occurs at the south pole, which in our coordinates on ${\mathbb R}^2$ corresponds to the origin. We will consider the case where the number of colliding vortices with degree $1$ exceeds the number with degree $-1$. The other case is handled similarly.

Thus we let $k$ denote the number of vortices with degree $-1$ and let $k+l$ denote the number  of vortices with degree $1$. We also denote by $I$ the set of indices $i \in \{1,2, ..., 2n\}$ corresponding to vortices involved in the collision at origin.
Using \eqref{ode} we calculate
\[
     \frac{d}{dt} r_i^2 = 2\mathbf p_i \cdot \dot{\mathbf p_i} = (1+r_i^2)^2 \left( \frac{r_i^2}{1+r_i^2} + d_i \sum_{j \neq i} d_j \frac{r_i^2 - \mathbf p_i \cdot \mathbf{p}_j}{|\mathbf{p}_i - \mathbf{p}_j|^2} \right).
\]
By the definition of $I$,
\[
    \lim_{t \rightarrow T_1} \mathbf p_i(t) \neq 0 \quad \mbox{for $i \notin I$.}
\]
Thus for $i \in I$ and $0<T_1 - t \ll 1$,
\[
    \frac{d}{dt} r_i^2 = \sum_{j \in I ; j \neq i} d_i d_j \frac{  r_i^2 -\mathbf{p}_i \cdot \mathbf{p}_j}{|\mathbf{p}_i - \mathbf{p}_j|^2} + o(1).
\]
Therefore,
\[
    \sum_{i \in I} \frac{d}{dt} r_i^2 = \sum_{i \in I} \sum_{j \in I ; j \neq i} d_i d_j \frac{  r_i^2 -\mathbf{p}_i \cdot \mathbf{p}_j}{|\mathbf{p}_i - \mathbf{p}_j|^2} + o(1) = \sum_{i,j \in I; i \neq j} d_i d_j + o(1).
\]
Note that in $I$ there are $k+l$ vortices with degree $1$ and $k$ vortices with degree $-1$. We easily check that
\[
    \sum_{i,j \in I; i \neq j} d_i d_j = 2[C_2^k + C_2^{k+l}- k(k+l)] \geq 2,
\]
where the last inequality follows from contradiction hypothesis:
\begin{equation} \label{Hp}
    C_2^k + C_2^{k+l}- k(k+l) \geq 1,
\end{equation}
and we use the convention that $C^0_2 = C^1_2 = 0$. Note that \eqref{Hp} is satisfied if $k=0$ since in this case $l \geq 2$. But this implies that there exists $i_* \in I$ such that $\frac{d}{dt} r_{i_*}^2 \geq \frac{2}{2k+l}$ for all $t$ sufficiently close to $T_1$, which contradicts the assumption that $\mathbf{p}_i \rightarrow 0$ as $t \rightarrow T_1$ for all $i \in I$.
\end{proof}

When two or more vortices collide at some time $T_0$, the system \eqref{ode} is no longer well-defined. However, we wish to consider a natural continuation of the flow after this collision time by restarting the flow with an appropriate removal or recombination of the colliding vortices. This motivates the definition below. We first make the assumption that if for some $\mathbf P_* \in {\mathbb S}^2$ and some $I \subset \{i \in {\mathbb N}: 1 \leq  i \leq 2n\}$, vortices $\{\mathbf P_i\}_{i \in I}$ collide at $\mathbf P_*$ when $t=T_*$, we must have
\begin{equation} \label{H}
    \sum_{i \in I} d_i \in \{-1,0,1\}.
\end{equation}
\begin{rk} \label{2v4v}
When there are $2$ vortices, it is clear that assumption \eqref{H} holds. For $4$ vortices, Proposition \ref{43} indicates that collisions cannot involve only vortices with degree $1$ or only $-1$. Thus assumption \eqref{H} is also satisfied in this case.
\end{rk}
\begin{defi} \label{passcoll}
We assume that if vortices $\{\mathbf P_i\}_{i \in I}$ collide at time $t=T_*$, then they annihilate each other if $\sum_{i \in I} d_i =0$. In this case, we restart \eqref{ode} for the remaining vortices with initial time $t=T_*$ and $\{\mathbf P_i\}_{i \in I}$ dropped. If $\sum_{i \in I} d_i = 1 (resp. -1)$, then we assume the colliding vortices $\{\mathbf P_i\}_{i \in I}$ combine to form a single vortex at $\mathbf P_*$ with degree $d_*=1 (resp. -1)$. Then we restart \eqref{ode} at time $t=T_*$ for the remaining vortices with $\{\mathbf P_i\}_{i \in I}$ dropped and $\mathbf P_*$ added.
\end{defi}
Using the definition above to extend the flow past collisions we may prove the following theorem that gives a clustering condition on initial data at $t=0$ that implies all vortices will eventually be annihilated. Here, for convenience, we assume the degrees $\{d_i\}_{i=1}^{2n}$ satisfy
\begin{equation} \label{degrees}
        d_i = \left\{\begin{array}{ll}
                        1 & \mbox{if $i$ is odd} \\
                       -1 & \mbox{if $i$ is even.}
                \end{array} \right.
\end{equation}

\begin{thm} \label{anh}
For fixed $0 \leq s <1$, let ${\mathcal A}_s = \{ (x,y,z) \in {\mathbb S}^2 : z \leq -\sqrt{1-s^2} \}$. Assume that vortices $\{{\mathbf P}_i \}_{i=1}^{2n}$ satisfy \eqref{odeS2} with $\mathbf P_i \in {\mathcal A}_s$ at $t=0$ for all $i$. Furthermore, assume that \eqref{H} holds at each collision time with the flow defined past collisions via Definition \ref{passcoll}. Then there are no vortices after a finite time if $\sqrt{1-s^2} > \frac{n-1}{n}$.
\end{thm}

\begin{proof}
Let
\[
    \mathbf V_0 := \mathbf P_{1} + \mathbf P_{2} + ...+\mathbf P_{2n}.
\]
Since $\mathbf P_i (0) \in {\mathcal A}$ for all $i$ and $s<1$, $|\mathbf V_0 (0)| \geq 2n \sqrt{1-s^2}>0$ . Then from Proposition \ref{ExpD}, since $\mathbf V_0(t) = \mathbf V_0(0)e^t$, there exists the first collision time $0<T_1 < \infty$. Assume that there exists a sequence of $m$ collision times $0<T_1<T_2<...<T_m<\infty$ with $m \geq 1$. Note that since each collision results in the annihilation of at least two vortices, $m \leq n$.

Since \eqref{H} holds at each collision time and the total degree of vortices is zero for all $t \geq 0$, there must always remain an even number of vortices in the restarted system of ODE's \eqref{ode} after the collision. Similarly, in light of Definition \ref{passcoll} and assumption \eqref{H}, we note that after the collision time $T_k$, an even number of vortices will be removed. For $k=1,2,...,m$, we define the total number of the removed vortices from $t=0$ to $t=T_k$ by $2j_k$. Recall that from \eqref{degrees}, the sign of the degree of a vortex $\mathbf P_i$ depends on the parity of the index $i$. Thus for each $1 \leq k \leq m$ we may relabel the indices of vortices such that at $t = T_k$,
\[
    \mathbf P_{2j_{k-1}+1} = \mathbf P_{2j_{k-1}+2}, \; \mathbf P_{2j_{k-1}+3} = \mathbf P_{2j_{k-1}+4} , \;..., \; \mathbf P_{2j_k-1} = \mathbf P_{2j_k},
\]
while
\[
    \mathbf P_i \neq \mathbf P_{i'} \mbox{ for } i \neq i', \; i,i'>2j_k,
\]
where $0=j_0 < j_1 < j_2 < ...< j_m \leq n$. Furthermore, for $t < T_k$
\[
    \mathbf P_{2j_k-1} \neq \mathbf P_{2j_k}.
\]
If $j_m = n$, then at $t = T_m$ all vortices have been annihilated. Therefore we will proceed by contradiction and assume that $j_m \leq n-1$. Now for each $1 \leq k \leq m$, set
\[
    \mathbf V_{k} := \mathbf P_{2j_k+1} + \mathbf P_{2j_k+2} + ...+\mathbf P_{2n},
\]
which is the sum of surviving vortices for $t \geq T_k$. We view each $\mathbf V_k$ as being defined on $[T_k, T_{k+1})$ with $\mathbf V_k(T_k)$ being a sum of vortices after appropriate removal of colliding vortices as in Definition \ref{passcoll}.
Then using the definition of $\mathbf V_0$ and $\mathbf V_{k}$, we have
\[
    \mathbf V_0 (T_1) = \mathbf U_{1} + \mathbf V_{1} (T_1), \mbox{ where } \mathbf U_{1} := \sum_{i=1}^{2j_1} \mathbf P_i (T_1),
\]
\[
    \mathbf V_{1} (T_2) = \mathbf U_{2} + \mathbf V_{2} (T_2), \mbox{ where } \mathbf U_{2} := \sum_{i=2j_1+1}^{2j_2} \mathbf P_i (T_2),
\]
and so on,
\[
     \mathbf V_{m-1} (T_m) = \mathbf U_{m} + \mathbf V_{m} (T_m), \mbox{ where } \mathbf U_{m} := \sum_{i=2j_{m-1}+1}^{2j_m} \mathbf P_i (T_m).
\]
Here for each $k$ we evaluated $\mathbf V_{k-1}(T_k)$ by taking the limit $t \rightarrow T_k^-$.
Note that after each collision time $T_k$, the restarted system of ODE's \eqref{ode} for remaining vortices $\{\mathbf P_i \}_{i=2j_k+1}^{2n}$ has the same structure as the original one i.e. Proposition \ref{ExpD} and Proposition \ref{Decay} hold for the restarted system. Hence from Proposition \ref{ExpD}, we have
\begin{equation} \label{ExpV}
    \frac{d }{dt} \mathbf V_{k-1} = \mathbf V_{k-1} \quad \mbox{for} \quad  t \in (T_{k-1}, T_k)
\end{equation}
for $1 \leq k \leq m$.
Applying \eqref{ExpV}, we derive
\begin{align}
    \mathbf V_{m} (T_m) &= \mathbf V_{m-1} (T_m) - \mathbf U_{m} \notag \\
    & =  \mathbf V_{m-1} (T_{m-1}) e^{T_m - T_{m-1}} - \mathbf U_{m} \notag \\
    & = [\mathbf V_{m-2}(T_{m-1}) - \mathbf U_{m-1}] e^{T_m - T_{m-1}} - U_{m} \notag \\
    & = [\mathbf V_{m-2} (T_{m-2}) e^{T_{m-1} - T_{m-2}} - \mathbf U_{m-1} ] e^{T_m - T_{m-1}} - \mathbf U_{m} \notag \\
    & = \mathbf V_{m-2} (T_{m-2}) e^{T_m - T_{m-2}} -  \mathbf U_{m-1} e^{T_m - T_{m-1}} - \mathbf U_{m}. \notag
\end{align}
Continuing this process finally we obtain
\begin{align} \label{vmfinal}
    \mathbf V_{m} (T_m) &= \mathbf V_0 (T_1) e^{T_m - T_1} - \sum_{k = 1}^m \mathbf U_{k} e^{T_m - T_k} \notag \\
    & = e^{T_m} [\mathbf V_0(0) - \sum_{k = 1}^m \mathbf U_{k} e^{-T_k} ].
\end{align}
Since $\sqrt{1-s^2} > \frac{n-1}{n}$,
\[
    |\sum_{k = 1}^m \mathbf U_{k} e^{-T_k} | \leq 2 j_m \leq 2(n-1) < 2n \sqrt{1-s^2} \leq |\mathbf V_0(0)|.
\]
Thus $|\mathbf V_{m} (T_m)| \neq 0$. But then from Proposition \ref{Decay}, there is at least one more collision after $t = T_m$ which contradicts the assumption that there are only $m$ collision times.
\end{proof}

\begin{cor}
Under the assumption of Theorem \ref{anh}, all vortices have been annihilated by the time $T = \ln \frac{1}{\kappa}$, where
\[
    \kappa := n \sqrt{1-s^2} - (n-1) >0.
\]
\end{cor}

\begin{proof}
Suppose $T_m$ is the last collision time, then $1 \leq j_{m-1} \leq n-1$. From \eqref{ExpV} and \eqref{vmfinal} we have
\begin{equation} \label{vm1}
    \mathbf V_{m-1} (T_m) = \mathbf V_{m-1} (T_{m-1}) e^{T_m - T_{m-1}},
\end{equation}
and
\begin{equation}
    \mathbf V_{m-1} (T_{m-1}) = e^{T_{m-1}} [\mathbf V_0 (0) - \sum_{k=1}^{m-1} \mathbf U_k e^{-T_k}].
\end{equation}
By the assumption of Theorem \ref{anh}, $|\mathbf V_0 (0)| \geq 2n \sqrt{1-s^2}$. On the other hand,
\[
    |\sum_{k=1}^{m-1} \mathbf U_k e^{-T_k}| \leq \sum_{k=1}^{m-1} |\mathbf U_k | \leq 2j_{m-1}.
\]
Thus
\begin{align} \label{vlowerbdd}
    |\mathbf V_{m-1} (T_{m-1})| & \geq e^{T_{m-1}} | \; |\mathbf V_0 (0)| - |\sum_{k=1}^{m-1} \mathbf U_k e^{-T_k}| \; | \notag \\
        & \geq e^{T_{m-1}} (2n \sqrt{1-s^2} - 2j_{m-1}) \notag \\
        & = 2 e^{T_{m-1}}  (\kappa + n -1 - j_{m-1}).
\end{align}
Substituting \eqref{vlowerbdd} into \eqref{vm1} we obtain
\[
    |\mathbf V_{m-1} (T_m)| \geq 2 e^{T_m} (\kappa + n -1 - j_{m-1}).
\]
Noting that $|\mathbf V_{m-1}| \leq 2(n-j_{m-1})$ and $\kappa <1$ we have
\begin{equation}
    T_m \leq \ln \left( \frac{n-j_{m-1}}{\kappa + n - 1 -j_{m-1}} \right) \leq \ln \frac{1}{\kappa}.
\end{equation}
\end{proof}

\begin{rk}
From Remark \ref{2v4v}, Theorem \ref{anh} indicates that all $\mathbf P_i$'s will vanish after a finite time if all $\mathbf P_i(0) \in {\mathcal A}_s$ for any $s<1$ when $n=1$, and any $s<\frac{\sqrt{3}}{2}$ when $n=2$. In fact, we have an explicit solution of \eqref{ode} when $n=1$: $\mathbf p_1 = (q,0)$, $\mathbf p_2 = (-q,0)$, where
\[
    q(t) = \sqrt{\frac{1+ce^t}{1-ce^t}},
\]
and
\[
               \left\{\begin{array}{ll}
                        c =0 & \mbox{if $q(0) =1$} \\
                        c>0  & \mbox{if $q(0) > 1$} \\
                        c<0  & \mbox{if $q(0) < 1$.} \\
                \end{array} \right.
\]
Note that $q(0) >1$ (resp. $q(0)<1$) implies that $\mathbf P_1(0)$ and $\mathbf P_2(0)$ are in the upper hemisphere (resp. lower hemisphere). Both cases lead to collision as predicted in Theorem \ref{anh}. When $n=2$, after the first collision time $T_1$, either all vortices are annihilated or there remain two vortices having different signs. In the latter scenario, the two-vortex flow then proceeds as described above.
\end{rk}

\section{Weighted Energy Identities for Ginzburg-Landau on ${S}^2$}
In this section we return to the PDE setting of \eqref{DGL} and we set ${\mathcal M} = {S}^2$. As a gradient flow, we of course
 know that solutions satisfy the standard dissipation rule
 \[
 \frac{d}{dt} E_{\e}(u)=- \int_{\mathcal M} \abs{u_t}^2\,dv_g,
 \]
 but our goal in this section is to derive weighted energy dissipation rules that we believe should have implications
 for vortex evolution on the $2$-sphere. As a by-product, we will derive necessary conditions on equilibrium solutions
 to \eqref{DGL} or \eqref{GP} that echo the symmetry requirements on vortex placement implied by Corollary \ref{centerofmass}.
 We wish to emphasize that, unlike the analysis in Sections 3 and 4, these results are not asymptotic in $\e$ but hold rather for
 any positive fixed $\e$.

 To this end, for $u: {S}^2 \times {\mathbb R}_+ \rightarrow \mathbb C$ we define the weighted energies $F_1,\,F_2$ and $F_3$  by
\begin{eqnarray*}
&&    F_1(u) := \int_{{S}^2} \left[ \frac{|\nabla_g u|^2}{2} + V(u) \right] \left(\frac{1-x_1}{2}\right) dv_g,\\
&&    F_2(u) := \int_{{S}^2} \left[ \frac{|\nabla_g u|^2}{2} + V(u) \right] \left(\frac{1-x_2}{2}\right) dv_g\quad\mbox{and}\\
&&    F_3(u) := \int_{{S}^2} \left[ \frac{|\nabla_g u|^2}{2} + V(u) \right] \left(\frac{1-x_3}{2}\right) dv_g,\\
\end{eqnarray*}
where $V(u) = \frac{(1-|u|^2)^2}{4\e^2}$ and we write the coordinates of any $\mathbf P \in {S}^2$ as $\mathbf P=(x_1,x_2,x_3)$. Again, since
the results to follow hold for any positive $\e$, we suppress the dependence of $\e$ in the notation for the weighted energies and in the solution $u$
to \eqref{DGL}.
Note that if $\mathbf p \in {\mathbb R}^2$ is the image of $\mathbf P$ via the stereographic projection mapping the north pole to $\infty$, then we have
\begin{equation}
     x_3 = \frac{r^2-1}{1+r^2}\quad\mbox{where}\;r^2=\abs{{\bf p}}^2.\label{north}
\end{equation}
Transforming $u$ via stereographic project (and still denoting it by $u$) we note that for example $F_3$ can be written as
\begin{equation} \label{wGE}
    F_3(u) = \int_{{\mathbb R}^2} \left[ \frac{|\nabla u|^2}{2} + \frac{4}{(1+r^2)^2}V(u) \right] w(r^2) dx,
\end{equation}
where $dx=dx_1\,dx_2$, $w(s) = \frac{1}{1+s}$ and $r^2 = |\mathbf p|^2$.
\begin{prop}\label{wee}
    Let $u$ be a solution to \eqref{DGL}. Then for $i=1,2$ and $3$, $F_i(u)$ satisfies
    \begin{equation}
         \frac{d}{dt} F_i(u) = -\int_{{S}^2}\left( \frac{1-x_i}{2}\right) |u_t|^2 + x_i V(u)\, dv_g.
    \end{equation}
\end{prop}
\begin{proof}
For simplicity of notation only in what follows we set $\e=1$. We will derive the identity for $F_3$. The other derivations
are identical.
Taking the derivative of \eqref{wGE} we obtain
\[
    \frac{d}{dt} F_3(u) = \int_{{\mathbb R}^2} \left[ \langle \nabla u_t, \nabla u \rangle  - \frac{4}{(1+r^2)^2} (1-|u|^2)u \cdot u_t \right] w(r^2)\, dx.
\]
Integrating by parts we have
\[
   \int_{{\mathbb R}^2} \langle \nabla u_t, \nabla u \rangle w(r^2) dx = -\int_{{\mathbb R}^2} w(r^2) u_t \cdot \Delta u + u_t \cdot \langle \nabla u, \nabla  w(r^2) \rangle\, dx.
\]
Thus
\begin{align} \label{dtdF}
    \frac{d}{dt} F(u) & = - \int_{{\mathbb R}^2} \frac{4 w(r^2)}{(1+r^2)^2} u_t \cdot \left[ \frac{(1+r^2)^2}{4} \Delta u + (1-|u|^2)u \right] + u_t \cdot  \langle \nabla u, \nabla  w(r^2) \rangle\, dx \notag \\
        & = - \int_{{\mathbb R}^2} \frac{4 w(r^2)}{(1+r^2)^2} |u_t|^2 + u_t \cdot  \langle \nabla u, \nabla  w(r^2) \rangle \, dx.
\end{align}
Here in the last equality we used the fact that from \eqref{DGL}$, u$ satisfies
\begin{equation} \label{utplane}
    u_t  = \frac{(1+r^2)^2}{4} \Delta u + (1-|u|^2)u \; \mbox{ in } {\mathbb R}^2.
\end{equation}
Since $\nabla w(r^2) = 2w'(r^2)x = -\frac{2x}{(1+r^2)^2}$, applying \eqref{utplane} again we have
\begin{align}
    \int_{{\mathbb R}^2} &  u_t \cdot  \langle \nabla u, \nabla  w(r^2) \rangle\, dx \notag \\
        &= \int_{{\mathbb R}^2} \left[ \frac{(1+r^2)^2}{4} \Delta u + (1-|u|^2)u \right] \cdot  2w'(r^2) \langle \nabla u, x \rangle\, dx \notag \\
        & = - \int_{{\mathbb R}^2} \left[ \frac{(1+r^2)^2}{4} \Delta u + (1-|u|^2)u \right] \cdot \frac{2}{(1+r^2)^2} \langle \nabla u, x \rangle \,dx. \notag
\end{align}
Then integrating by parts twice we find
\begin{align} \label{utgugw}
     \int_{{\mathbb R}^2} &  u_t \cdot  \langle \nabla u, \nabla  w(r^2) \rangle \notag \\
        = & \int_{{\mathbb R}^2} \frac{1}{2} \langle \nabla u, \nabla \langle \nabla u, x \rangle \rangle +  \frac{2}{(1+r^2)^2} \langle \nabla V(u) , x\rangle\, dx \notag \\
        = & \int_{{\mathbb R}^2} \frac{1}{2} \left( |\nabla u|^2 +\frac{1}{2} \langle \nabla |\nabla u|^2, x \rangle \right) +  \frac{2}{(1+r^2)^2} \langle \nabla V(u) , x\rangle\, dx \notag \\
        = & -  \int_{{\mathbb R}^2} \left[ \frac{4}{(1+r^2)^2} - \frac{8r^2}{(1+r^2)^3} \right] V(u)  \, dx \notag \\
        = & \int_{{\mathbb R}^2} \left( \frac{r^2 -1}{ 1+r^2}\right) V(u)  \frac{4}{(1+r^2)^2}\, dx.
\end{align}
Combining \eqref{north}, \eqref{dtdF} and \eqref{utgugw} we obtain
\begin{align}
    \frac{d}{dt} F_3(u) & =  - \int_{{\mathbb R}^2} \left[ w(r^2)|u_t|^2 + \left(\frac{r^2 -1}{ 1+r^2}\right) V(u)  \right] \frac{4}{(1+r^2)^2}\, dx \notag \\
        & = -\int_{{S}^2} \left(\frac{1-x_3}{2}\right) |u_t|^2 + x_3 V(u) \, dv_g.
\end{align}
\end{proof}
We view Proposition \ref{wGE} as a tool for studying vortex annihilation for fixed $\e$. Note for example, that if all vortices initially
reside in the first quadrant so that $V(u(x,0))$ is significantly larger in the first quadrant than in all others, one would have
\[
\int_{{S}^2} x_i V(u) dv_g>0\quad\mbox{at time}\;t=0\quad\mbox{for}\;i=1,2,3.
\]
Hence, along with the standard energy $E_{\e}$, by Proposition \ref{wGE}, the three weighted energies would also dissipate for some positive interval of time.
This could provide the basis for an annihilation result for \eqref{DGL} analogous to Theorem \ref{anh} for the ODE flow. Carrying out such an argument remains
work in progress.

Finally, we note the following:
\begin{cor}\label{symmetry}
Any stationary solution $u$ to \eqref{DGL}, or equivalently, to \eqref{GP} on $S^2$ must satisfy the first moment identities
    \begin{equation}
         \int_{{S}^2} x_1V(u)\, dv_g = \int_{{S}^2} x_2V(u)\, dv_g = \int_{{S}^2} x_3 V(u)\, dv_g =0.\label{symS2}
    \end{equation}
\end{cor}
One can view these conditions as in some sense balance laws for the placement of vortices in critical points of Ginzburg-Landau on $S^2$.
In this light, they can be compared with the balance condition derived in Corollary \ref{centerofmass} for collections of vortices representing critical points of the renormalized
energy W (cf. \eqref{rne}) on $S^2$. It remains a challenging open problem to establish a type of symmetry result for the vortices of critical points on $S^2$ but
\eqref{symS2} at least points in this direction.

\end{document}